\documentclass[12pt, reqno, a4paper]{amsart}
\usepackage{graphicx, epsfig, psfrag}
\usepackage{amsfonts,amsmath,amssymb,amsbsy,amsthm}
\usepackage{bm}
\usepackage{color}
\usepackage{float}
\usepackage{hyperref}
\usepackage{mathrsfs}
\usepackage{multirow}
\usepackage{constants}

\usepackage{pdfsync}

\usepackage{environments}

\numberwithin{equation}{section}
\numberwithin{theorem}{section}
\numberwithin{remark}{section}

\renewcommand{\paragraph}[1]{\subsubsection{#1}}

\renewcommand{\cases}[1]{\left\{ \begin{array}{rl} #1 \end{array} \right.}
\newcommand{\smfrac}[2]{{\textstyle \frac{#1}{#2}}}

\def\XXint#1#2#3{{\setbox0=\hbox{$#1{#2#3}{\int}$ }
\vcenter{\hbox{$#2#3$ }}\kern-.6\wd0}}

\def\b{\big}
\def\B{\Big}
\def\bg{\bigg}

\def\sep{\,|\,}
\def\bsep{\,\b|\,}

\def\conv{{\rm conv}}

\def\id{{\rm id}}

\def\R{\mathbb{R}}
\def\N{\mathbb{N}}
\def\Z{\mathbb{Z}}

\def\WW{{\rm W}}
\def\CC{{\rm C}}

\def\LL{{\rm L}}

\def\dx{\,{\rm d}x}

\def\pp{\partial}

\def\<{\langle}
\def\>{\rangle}

\def\mA{{\sf A}}

\def\mF{{\sf F}}
\def\mG{{\sf G}}

\def\mI{{\sf I}}
\def\mJ{{\sf J}}
\def\mQ{{\sf Q}}

\def\eps{\varepsilon}

\def\ac{{\rm ac}}
\def\a{{\rm a}}
\def\c{{\rm c}}
\def\i{{\rm i}}

\newcommand{\Da}[1]{D_{#1}}
\newcommand{\Dc}[1]{\partial_{#1}}
\def\D{\partial}
\def\del{\delta}

\def\Is{\mathcal{I}}

\def\As{\mathcal{A}}
\def\Cs{\mathcal{C}}

\def\L{\mathcal{L}}

\def\Us{\mathscr{U}}

\def\Fs{\mathscr{F}}
\def\Fsc{\Fs_\Cs}
\def\Fsa{\Fs_\As}
\def\Fsi{\Fs_\Is}
\def\Ys{\mathscr{Y}}

\def\yF{y_\mF}

\def\E{\mathscr{E}}
\def\Ea{\E_\a}
\def\Ec{\E_\c}

\def\T{\mathscr{T}}
\def\Tc{\T_\Cs}
\def\Ta{\T_\As}
\def\Ti{\T_\Is}

\def\PO{{\rm P}_0}

\def\PI{{\rm P}_1}

\def\CR{{\rm N}_1}

\def\Eac{\E_\ac}

\def\Sa{\Sigma_\a}
\def\Sc{\Sigma_\c}
\def\Sac{\Sigma_\ac}
\def\Sacm{\widehat{\Sigma}_\ac}

\def\RO{\mathcal{R}}

\def\Vc{V^\c}
\def\Vi{V^\i}
\def\Vac{V^\ac}

\def\Vs{\mathscr{V}}

\definecolor{cocol}{rgb}{0.7, 0, 0}
\definecolor{lzcol}{rgb}{0, 0, 0.7}

\begin{document}

\title[]{Construction and sharp consistency estimates for
  atomistic/continuum coupling methods with general interfaces: a 2D
  model problem}

\author{C. Ortner}
\address{C. Ortner\\ Mathematics Institute \\ Zeeman Building \\
  University of Warwick \\ Coventry CV4 7AL \\ UK}
\email{c.ortner@warwick.ac.uk}

\author{L. Zhang}
\address{L. Zhang \\ Mathematical Institute\\
  24-29 St Giles' \\ Oxford OX1 3LB \\ UK}
\email{zhang@maths.ox.ac.uk}

\date{\today}

\thanks{This work was supported by EPSRC Grant ``Analysis of
  Atomistic-to-Continuum Coupling Methods'' and the EPSRC Critical
  Mass Programme ``New Frontiers in the Mathematics of Solids''
  (OxMoS).}



\begin{abstract}
  We present a new variant of the geometry reconstruction approach for
  the formulation of atomistic/continuum coupling methods (a/c
  methods). For multi-body nearest-neighbour interactions on the 2D
  triangular lattice, we show that patch test consistent a/c methods
  can be constructed for arbitrary interface geometries. Moreover, we
  prove that all methods within this class are first-order consistent
  at the atomistic/continuum interface and second-order consistent in
  the interior of the continuum region.
\end{abstract}

\maketitle


\section{Introduction}
Atomistic/continuum coupling methods (a/c methods) are a class of
coarse-graining techniques for the efficient simulation of atomistic
systems with localized regions of interest interacting with long-range
elastic effects that can be adequately described by a continuum
model. We refer to \cite{Miller:2008}, and references therein, for an
introduction and discussion of applications.

In the present work we are concerned with the construction and
rigorous analysis of energy-based a/c methods in a 2D model
problem. Our starting point is the geometry reconstruction approach
proposed by Shimokawa {\it et al}~\cite{Shimokawa:2004} and by E, Lu
and Yang \cite{E:2006} for the construction of ``consistent'' a/c
methods in 2D and 3D. We propose a new variant of that approach to
define a modified site potential at the a/c interface, which has
several free parameters. We then ``fit'' these parameters so that the
resulting a/c hybrid energy satisfies an energy consistency condition
and a force consistency condition (see \eqref{eq:energy_cons} and
\eqref{eq:force_cons} for the precise definition of these terms; in
the terminology of quasicontinuum methods our hybrid energy is free of
ghost forces).

Explicit constructions along these lines can be found in
\cite{Shimokawa:2004} for pair potentials and in \cite{E:2006} for
coupling a finite-range multi-body potential to a nearest-neighbour
potential, for high-symmetry interfaces. Our focus in the present work
is the coupling to a continuum model and interfaces with corners; both
of these cases are only briefly touched upon in \cite{E:2006}.

In recent years there has been considerable activity in the numerical
analysis literature on the classification and rigorous analysis of a/c
methods (see \cite{Dobson:2008b, Dobson:qce.stab, emingyang,
  Ortner:qnl.1d, OrtnerWang:2009a} and references therein). Much of
this work has been restricted to one-dimensional problems; only very
recently some progress has been made on the analysis of a/c methods in
2D and 3D \cite{LuMing:2011pre, Ortner:2011:patch, OrtShap:2011a}.

The first rigorous error estimates for the method proposed in
\cite{E:2006} (together with a wider class of related methods), in
more than one dimension, are presented in \cite{Ortner:2011:patch} for
2D finite range multi-body interactions. The work
\cite{Ortner:2011:patch} {\it assumes} the existence of an interface
potential so that the resulting a/c energy satisfies certain energy
and force consistency conditions (a variant of the {\it patch test})
and then established first-order consistency of the resulting a/c
method in negative Sobolev norms.

Several important questions remain open: 1. It is yet unclear whether
constructions of the type proposed in \cite{E:2006, Shimokawa:2004}
can be carried out for interfaces with corners. 2. The error
estimates in \cite{Ortner:2011:patch} contain certain non-local terms
that enforce unnatural assumptions (e.g., connectedness of the
atomistic region). 3. Moreover, this nonlocality causes suboptimal
error estimates; namely, it destroys the second-order consistency of
the Cauchy--Born model (see, e.g., \cite{Dobson:2008b, E:2007a,
  Ortner:qnl.1d}), and an unnatural dependence of the interface width
enters the error estimates. (Moreover, we note that the error
estimates in \cite{OrtShap:2011a} for a different a/c method are only
first-order as well.)

The purpose of the present work is to investigate for a model problem
whether these restrictions are genuine, or of a technical nature.  To
that end we formulate an atomistic model on the 2D triangular lattice
with nearest-neighbour multi-body interactions (effectively these are
third neighbour interactions), and construct new a/c methods in the
spirit of \cite{E:2006, Shimokawa:2004}. We then prove that the
resulting methods are all first-order consistent in the interface
region and second-order consistent in the interior of the continuum
region, which is the first generalisation of the optimal
one-dimenional result \cite[Theorem 3.1]{Ortner:qnl.1d} to two
dimensions.

Although it may seem restrictive at first glance to consider only
nearest-neighbour potentials, we note that this is in fact an
important case to consider. For example, bond-angle potentials (which
are included in our analysis) usually consider only angles between
nearest-neighbour bonds. More generally, multi-body effects are
usually restricted to very small interaction neighbourhoods, while
long-range effects are often only displayed in pair potentials (in
particular, Lennard-Jones and Coulomb), which can be treated, for
example, using Shapeev's method~\cite{OrtShap:2011a, Shapeev:2010a,
  Shapeev:2011a}.


\section{Atomistic Continuum Coupling}
\label{sec:stress}

\subsection{Atomistic model}
\label{sec:flat:a}
We consider a nominally infinite crystal, but restrict admissible
displacements to those with compact support. Thus we avoid any
discussion of boundary conditions, which are unimportant for the
purpose of this work.

%
Let $\mQ_6$ denote a rotation through arclength $\pi/3$. As a
reference configuration we choose the triangular lattice (see also
Figure \ref{fig:lattice}):
\begin{align*}
  \L := \mA \Z^2, \qquad \text{where } &~\mA := (a_1, a_2), \\
  &~a_1 := (1, 0)^\top, \text{ and } a_j := \mQ_6^{j-1} a_1, j \in \Z.
\end{align*}
We will frequently use the following relationships between the vectors
$a_j$:
\begin{displaymath}
  a_{j+6} = a_{j}, \quad a_{j+3} = -a_j, \quad \text{and} \quad
  a_{j-1} + a_{j+1} = a_j \qquad \text{for all } j \in \Z.
\end{displaymath}
For future reference we also define $\mathbf{a} := (a_j)_{j=1}^6$, and
$\mF \mathbf{a} := (\mF a_j)_{j = 1}^6$, for $\mF \in \R^{2 \times 2}$.

\begin{figure}
  \begin{center}
    \includegraphics[height=3.5cm]{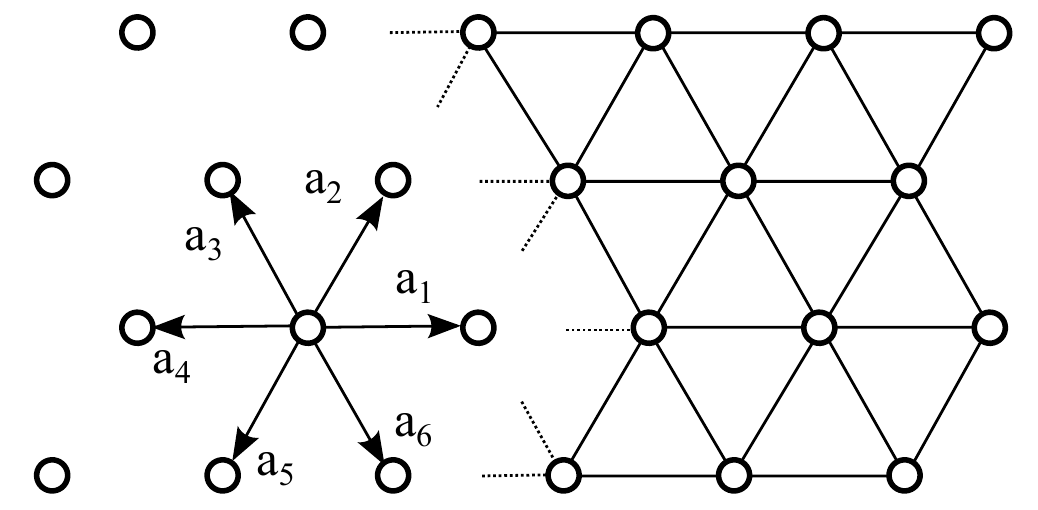}
  \end{center}
  \caption{ \label{fig:lattice} The 2D triangular lattice and its
    canonical triangulation. }
\end{figure}

Our choice of reference configuration is largely motivated by the fact
that $\L$ possesses a canonical triangulation (see Figure
\ref{fig:lattice}, and \S\ref{sec:flat:cb}), which will be convenient
in our analysis.

The set of displacements and deformations with compact support are
given, respectively, by
\begin{align*}
  \Us_0 :=~& \big\{ u : \L \to \R^2 : u(x) \neq 0  \text{ for at
    most finitely many } x \in \L \big\}, \quad \text{and} \\
  \Ys_0 :=~& \big\{ y : \L \to \R^2 : y - {\rm id} \in \Us_0 \big\}.
\end{align*}
We remark that deformations are usually required to be at least
invertible, but that we avoid this requirement by making simplifying
assumptions on the interaction potential.

A homogeneous deformation is a map $\yF : \L \to \R^2$, $\yF(x) := \mF
x$, where $\mF \in \R^{2 \times 2}$. We note that $\yF \notin \Ys_0$
unless $\mF = \mI$.

For a map $v : \L\to\R^k$, $k \in \N$, we define the forward finite
difference operator
\begin{displaymath}
\Da{j} v(x) := v(x+ a_j) - v(x), \quad x \in \L, j \in \Z,
\end{displaymath}
and we define the family of all nearest-neighbour finite differences
as $\Da{}y(x) := (\Da{j} y(x))_{j = 1}^6$.

We assume that the atomistic interaction is described by a
nearest-neighbour multi-body site energy potential $V \in \CC^3(\R^{2
  \times 6})$, with $V({\bf a}) = 0$, so that the energy of a
deformation $y \in \Ys_0$ is given by
\begin{displaymath}
  \Ea(y) := \sum_{x \in \L} V\big( \Da{} y(x) \big).
\end{displaymath}
The assumption $V({\bf a}) = 0$ guarantees that $\Ea(y)$ is finite for
all $y \in \Ys_0$.

\subsection{The Cauchy--Born approximation}
\label{sec:flat:cb}
For deformation fields $y \in \WW^{1,\infty}(\R^2; \R^2)$, such that
$y - \id$ has compact support, we define the Cauchy--Born energy
functional
\begin{displaymath}
  \Ec(y) := \int_{\R^2} W(\D y) \dx, \qquad \text{where} \quad
  W(\mF) := \smfrac{1}{\Omega_0}V\big( \mF \mathbf{a} \big),
\end{displaymath}
$W \in \CC^3(\R^{2 \times 2}; \R)$, is the {\em Cauchy--Born stored
  energy function}.
The factor $\Omega_0:=\sqrt{3}/2$ is the volume of one primitive cell
of $\L$, that is, $W(\mF)$ is the energy per unit volume of the
lattice~$\mF \L$.


If $y \in \Ys_0$ is a {\em discrete} deformation, then we define its
Cauchy--Born energy through piecewise affine interpolation: The
triangular lattice $\L$ has a canonical triangulation $\T$ into closed
triangles depicted in Figure \ref{fig:lattice}.
Henceforth, we shall always identify a function $v : \L \to \R^k$ with
its $\PI$-interpolant, which belongs to
$\WW^{1,\infty}(\R^2;\R^k)$. For a discrete deformation $y \in \Ys_0$,
we can then write the Cauchy--Born energy as
\begin{equation}
  \label{eq:defn_Ec_W}
  \Ec(y) = \int_{\R^2} W(\D y) \dx 
  = \sum_{T \in \T} |T| W(\Dc{T} y),
\end{equation}
where we define $\Dc{T} y := \D y(x)|_{x \in T}$ and note that $|T| =
\Omega_0/2$ for all triangles $T \in \T$.

Note that $W(\mI) = 0$ and hence $\Ec(y)$ is finite for all $y \in
\Ys_0$.

Alternatively, $\Ec$ can be written in terms of site energies, which
will be helpful for the definition of a/c methods. Each vertex $x \in
\L$ has six adjacent triangles, which we denote by $T_{x, j} :=
\conv\{ x, x+a_j, x+a_{j+1}\}$, $j = 1,\dots, 6$ (cf. Figure
\ref{fig:Txj_xTj}). With this notation, 
\begin{equation}
  \label{eq:defn_Vc}
  \Ec(y) = \sum_{x \in \L} \Vc(\Da{} y(x)), \qquad \text{where} \quad
  \Vc(\Da{}y(x)) := \frac{\Omega_0}{6} \sum_{j = 1}^6 W(\Dc{T_{x,j}}
  y).
\end{equation}
Note that $\Vc \in \CC^3(\R^{2 \times 6})$ is well-defined since
$\Dc{T_{x,j}} y$ is determined by the finite differences $\Da{j} y(x)$
and $\Da{j+1} y(x)$.

\subsection{A/c coupling via geometry reconstruction}
\label{sec:flat:gcqc}
Let $\As \subset \L$ denote the set of all lattice sites for which we
require full atomistic accuracy. We denote the set of interface
lattice sites by
\begin{displaymath}
  \Is := \b\{ x \in \L \setminus \As \bsep x + a_j \in \As \text{ for some } j \in \{1,\dots,6\}\b\},
\end{displaymath}
and we denote the remaining lattice sites by $\Cs := \L \setminus (\As
\cup \Is)$; cf. Figure \ref{fig:acmethod}.

\begin{figure}
  \begin{center}
    \includegraphics[height=3cm]{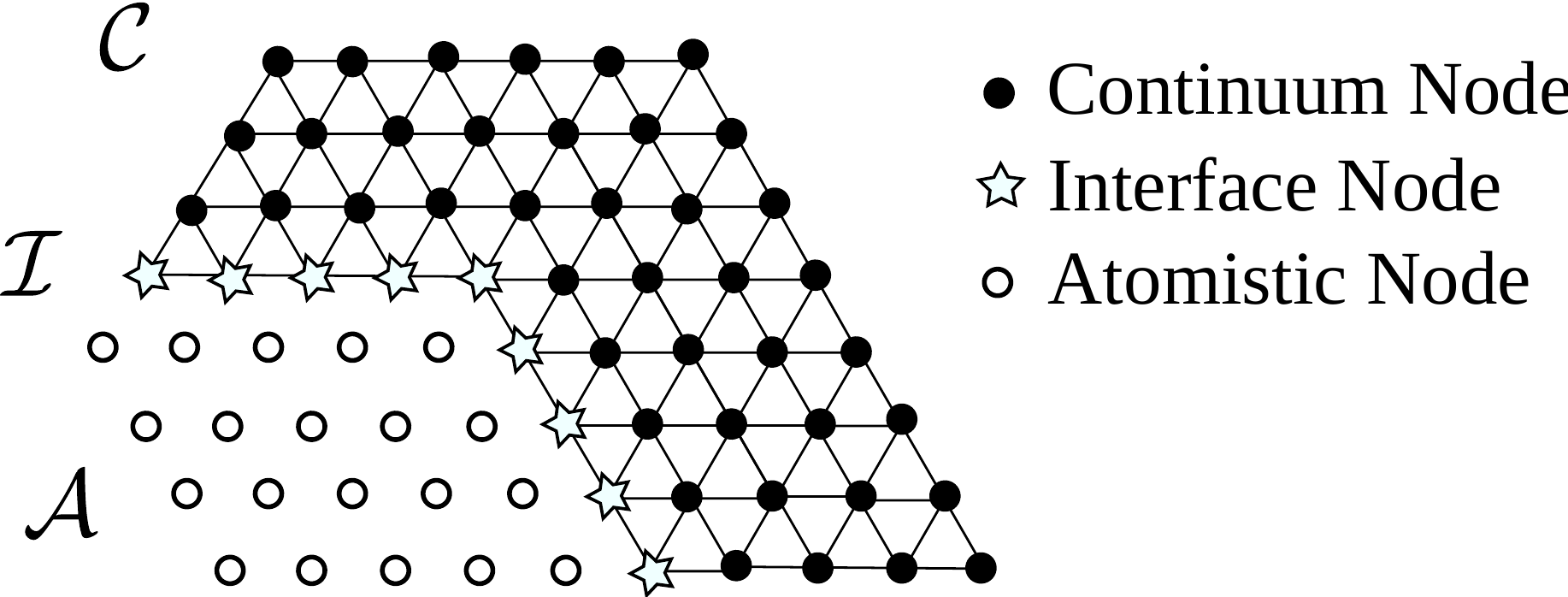}
   \end{center}
   \caption{Atomistic-interface-continuum domain decomposition.}
   \label{fig:acmethod}
\end{figure}

A general form for the constuction of a/c coupling energies is
\begin{equation}
  \label{eq:Eac_general}
  \Eac(y) = \sum_{x \in \As} V(\Da{}y(x)) + \sum_{x \in \Is}
  \Vi_x(\Da{}y(x)) + \sum_{x \in \Cs} \Vc(\Da{}y(x)),
\end{equation}
where $\Vi_x, x \in \Is$, are the interface site potentials that
define the method (the atomistic site potential and the continuum site
potential are determined by the atomistic model).

For example, if we choose $\Vi_x = V$, then we obtain the original
{\em quasicontinuum method} \cite{Ortiz:1995a} (the QCE method). It is
well understood that the QCE method suffers from the occurance of
ghost forces, which result in large modelling errors
\cite{Dobson:2008b, Miller:2008, emingyang, OrtnerWang:2009a,
  Shenoy:1999a}. 

In the following we present a new variant of the geometry
reconstruction approach \cite{E:2006, Shimokawa:2004}  for constructing $\Vi$.
%
%
We define the interface potential as
\begin{equation}
  \label{eq:VI_gc}
  \Vi_x(\Da{}y(x)) := V(\RO_x\Da{}y(x)),
\end{equation}
where $\RO_x$ is a {\em geometry reconstruction operator} of the
general form
\begin{equation}
  \label{eq:flat:reconop}
  \RO_x \Da{} y(x) := \b( \RO_x \Da{j} y(x) \b)_{j = 1}^6, \quad
  \text{and} \quad
  \RO_x \Da{j}y(x) :=\sum_{i = 1}^6 C_{x,j,i}\Da{i}y(x).
\end{equation}
Here $(C_{x, j, i})_{j, i = 1}^6$, $x \in \Is$, are free parameters
of the method that can be determined to improve the accuracy of the
coupling scheme.

We use the acronym ``GR-AC method'' (geometry reconstruction-based
atomistic-to-continuum coupling method) to describe methods of the
type \eqref{eq:Eac_general} where the interface site potential is of
the form \eqref{eq:VI_gc}.

We aim to determine parameters $C_{x,j,i}$ such that the coupling
energy $\Eac$ satisfies the following conditions, which we label,
respectively, {\it local energy consistency} and {\it local force
  consistency}:
\begin{align}
  \label{eq:energy_cons}
  \Vi_x( \mF {\bf a} ) = V( \mF {\bf a} )~& \qquad \forall \mF \in
  \R^{2 \times 2}, \quad \forall x \in \Is, \qquad \text{and} \\
  \label{eq:force_cons}
  f_\ac( x; \yF ) = 0 ~& \qquad \forall \mF \in \R^{2 \times 2},
  \quad \forall x \in \L,
\end{align}
where $f_\ac(x;y)$ is the force acting on the atom at site $x$,
initially defined by
\begin{displaymath}
  f_\ac(x; y) := - \frac{\pp \Eac(y)}{\pp y(x)} \in \R^2 \qquad
  \text{for } y \in \Ys_0;
\end{displaymath}
however, we immediately see that $f_\ac$ involves only a sum over a
finite set of lattice sites, and hence the formula can be extended to
all maps $y : \L \to \R^2$. In particular, \eqref{eq:force_cons} is a
well-posed condition. Taken together, we call \eqref{eq:energy_cons}
and \eqref{eq:force_cons} the {\em patch test}. A hybrid energy $\Eac$
of the form \eqref{eq:Eac_general} is called {\em patch test consistent}
if it satisfies both conditions.

In the remainder of the paper, we will determine choices of the
parameters $C_{x,j,i}$ for general a/c interface geometries that give
patch test consistent coupling methods. Moreover, we will prove that
for all parameter choices we determine, the resulting a/c method is
first-order consistent at the interface and second-order consistent in
the interior of the continuum region. This extends the optimal 1D
result in \cite{Ortner:qnl.1d}.

\begin{remark}
  1. To obtain a method with improved complexity one should use a
  coarser finite element discretisation in the continuum region. It
  was seen in \cite{OrtnerWang:2009a, Ortner:2011:patch} that the
  coarsening step can be understood using standard finite element
  methodology, and hence we focus only on the modification of the
  model, and the resulting {\em modelling errors}.

  2.  Realistic interaction potentials have singularities for
  colliding nuclei, i.e., for deformations that are not
  injective. Clearly, our assumption that $V \in \CC^3(\R^{2 \times
    6})$ contradicts this. It is conceptually easy to admit more
  general site potentials in our work, however, this would introduce
  additional technical steps that are of little relevance to the
  problems we wish to study.
\end{remark}

%
%

\subsection{Additional assumptions and notation}
We use $|\cdot|$ to denote the $\ell^2$-norm on $\R^n$, and the
Frobenius norm on $\R^{n \times m}$. Generic constants that are
independent of the potential (and the constants defined in the
following paragraphs) and the underlying deformations are denoted by
$c$. Although it is possible in principle to trace all constants in
our proofs, it would require additional non-trivial computations to
optimize them.

\subsubsection{Properties of $V$}
\label{sec:prop_V}
We define notation for partial derivatives of $V$, for ${\bf g} \in
\R^{2 \times 6}$, as follows:
\begin{displaymath}
  \partial_j V(\mathbf{g}):=\frac{\pp V(\mathbf{g})}{\pp g_j} \in\R^2,
  \quad\text{and}\quad
  \partial_{i,j}V(\mathbf{g}):=\frac{\pp^2 V(\mathbf{g})}{\pp g_i\pp
    g_j} \in\R^{2\times 2}, 
  \quad\text{for } i, j \in \{1,\dots,6\},
\end{displaymath}
and similarly, the third derivative $\pp_{i,j,k} V({\bf g}) \in
\R^{2\times 2 \times 2}$, which we will never use explicitly. We will
frequently also use the short-hand notation
\begin{displaymath}
  V_{x,j} := \pp_j V(Dy(x)), \quad V_{T, j} := \pp_j V((\Dc{T} y) {\bf
    a}), \quad \text{and} \quad
  V_{\mF, j} := \pp_j V(\mF {\bf a}),
\end{displaymath}
as well as analogous notation for second derivatives and for the site
potentials $\Vc$, $\Vi$, and for $\Vac$, which is defined in
\eqref{eq:defn_Vac}.

Interpreting the second and third partial derivatives as
multi-linear forms we define the global bounds
\begin{align*}
  M_2 :=~& \sum_{i,j = 1}^6 
  \sup_{{\bf g} \in \R^{2 \times 6}} \sup_{\substack{h_1, h_2
      \in \R^2 \\ |h_1| = |h_2| = 1}} \pp_{i,j} V({\bf g})[h_1, h_2],
    \quad \text{and} \\[-2mm]
  M_3 :=~& \sum_{i,j,k = 1}^6 
  \sup_{{\bf g} \in \R^{2 \times 6}} \sup_{\substack{h_1, h_2, h_3
      \in \R^2 \\ |h_1| = |h_2| = |h_3| = 1}} \pp_{i,j,k} V({\bf g})[h_1, h_2, h_3].
\end{align*}
With this notation it is straightforward to show that
\begin{equation}
  \label{eq:Lip_partialV}
  \sum_{i = 1}^6 \b|\pp_iV(\mathbf{g})-\pp_iV(\mathbf{h})\b| \leq
  M_2 \max_{j = 1, \dots, 6} |g_j - h_j|, \qquad \text{ for } {\bf g},
  {\bf h} \in \R^{2 \times 6}.
\end{equation}

We also assume that $V$ satisfies the point symmetry
\begin{equation}
  \label{eq:pt_symm}
  V\b( (-g_{j+3})_{j = 1}^6 \b) = V( {\bf g} ) \qquad \forall {\bf g}
  \in \R^{2 \times 6}.
\end{equation}
The following identities are immediate consequences of this condition:
\begin{align}
  \label{eq:pt_symm_D}
  \pp_{i}V(\mF{\bf a}) =~& - \pp_{i+3} V(\mF{\bf a}), \qquad \text{for } i =
  1, \dots, 6,  \quad \mF \in \R^{2 \times 2}\\
  \label{eq:pt_symm_D2}
  \pp_{ij} V(\mF{\bf a}) =~& \pp_{i+3,j+3} V(\mF{\bf a}), \qquad
  \text{for } i, j = 1, \dots, 6, \quad \mF \in \R^{2 \times 2}.
\end{align}
We will prove results on the class $\Vs$, of all site potentials that
satisfy \eqref{eq:pt_symm}, 
\begin{displaymath}
  \Vs := \b\{ V \in \CC^3(\R^{2\times 6}) \bsep \text{ $V$ satisfies
    \eqref{eq:pt_symm}}\, \b\}.
\end{displaymath}

We will frequently use the following shorthand notation for partial
derivatives of $V$, when there is no ambiguity in their meaning:
\begin{displaymath}
  V_{x, j} := \pp_j V(Dy(x)), \qquad
  V_{\mF, j} := \pp_j V(\mF{\bf a}), \qquad
  V_{T, j} := V_{\Dc{T} y, j},
\end{displaymath}
and analogous symbols for other potentials that we will
introduce throughout the text. 

\subsubsection{Linear functionals}
For $y \in \Ys_0$ and $u \in \Us_0$ we denote the directional derivative
of $\Ea$ by
\begin{displaymath}
  \b\< \del\Ea(y), u \b\> := \lim_{t \to 0} \frac{\Ea(y+t u) - \Ea(y)}{t}.
\end{displaymath}
We call $\del\Ea(y)$ the {\it first variation} of $\Ea$ and understand
it as an element of $\Us_0^*$. We use analogous notation for other
functionals. This paper is largely concerned with establishing bounds
on the {\it modelling error} $\del\Ea(y) - \del\Eac(y)$.

To obtain sharp error estimates in $\WW^{1,p}$-like norms, one needs
to bound modelling errors in negative Sobolev norms, or, in our case,
discrete verions thereof. Let $\ell : \Us_0 \to \R$ be a linear
functional, and let $\frac1p + \frac{1}{p'} = 1$, $1 \leq p, p' \leq
\infty$, then we define
\begin{displaymath}
  \| \ell \|_{\Us^{-1,p}} := \sup_{\substack{u \in \Us_0 \\ \| \D u \|_{\LL^{p'}} =1}} 
  \b\<\ell, u \b\>.
\end{displaymath}

\subsubsection{Notation for the lattice and the triangulation}
$\L$ is the set of vertices of $\T$, and we denote the set of edges of
$\T$ by $\Fs$, with edge midpoints $m_f$, $f \in \Fs$.

For each vertex $x \in \L$ and direction $a_j$, let $T_{x,j} :=
\conv\{ x, x+a_j, x+a_{j+1}\} \in \T$, $j = 1,\dots,6$ (see
Figure~\ref{fig:Txj_xTj}). The edge $(x, x+a_j)$ is the intersection
of the two elements $T_{x,j}$ and $T_{x, j-1}$. Moreover, let $x_{T,
  j} \in \L$ be the unique lattice point so that both $x_{T,j},
x_{T,j}+a_j \in T$ (again, see Figure \ref{fig:Txj_xTj}).
%

\begin{figure}
  \begin{center}
    \includegraphics[height= 0.2\textwidth]{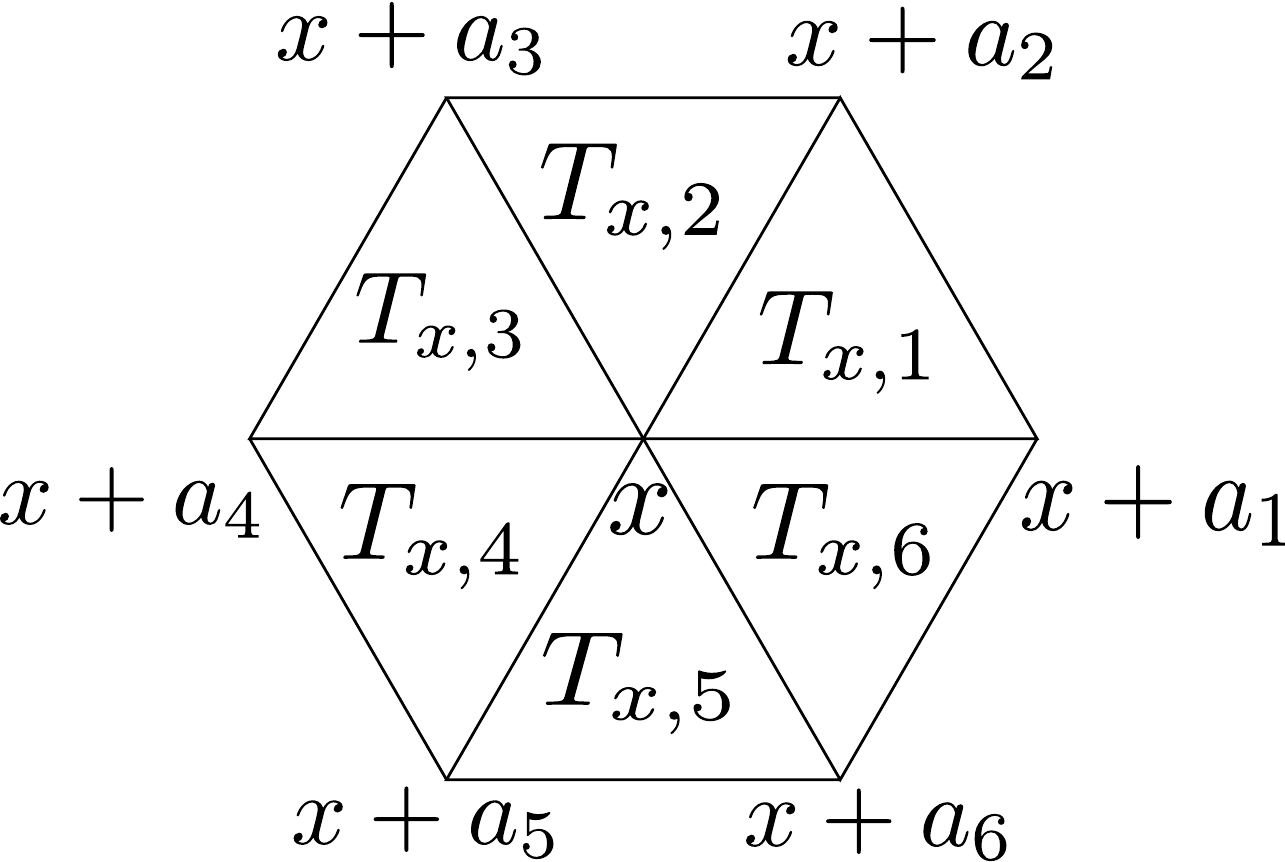} \qquad 
    \includegraphics[height= 0.15\textwidth]{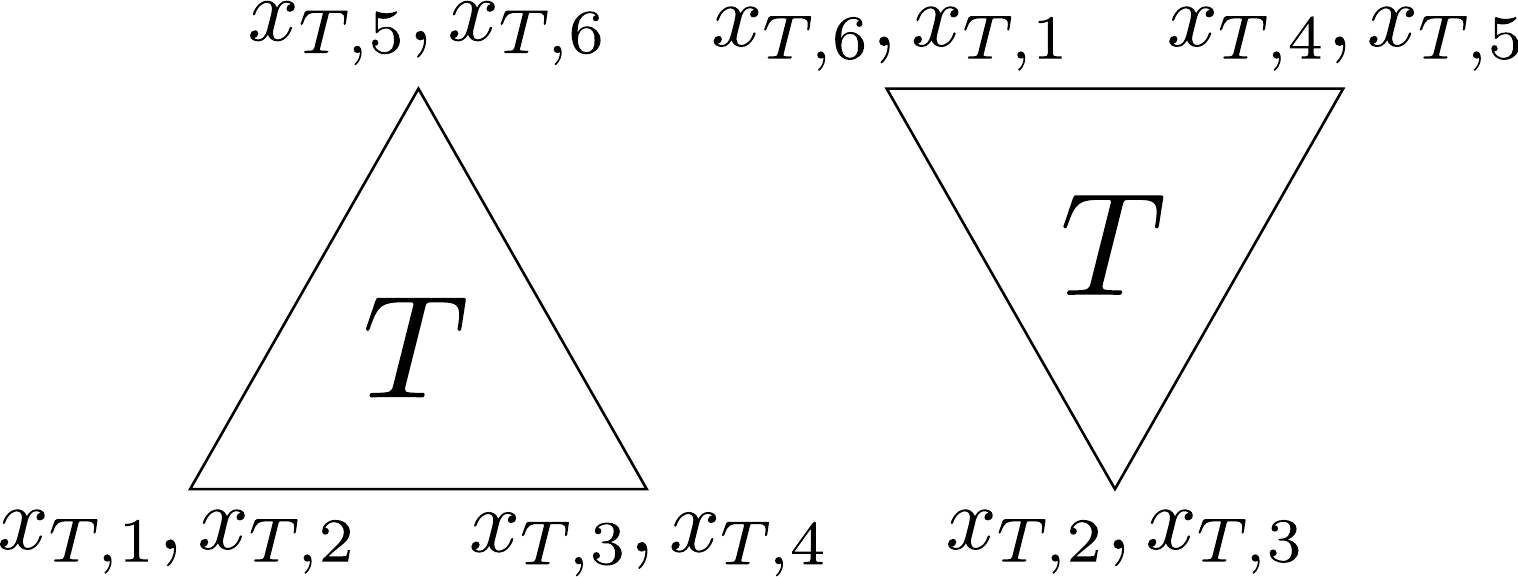}
   \end{center}
   \caption{Convention for the symbols $T_{x,j}$ and $x_{T,j}$.}
  \label{fig:Txj_xTj}
\end{figure}

\subsubsection{Discrete regularity}
To measure regularity or ``smoothness'' of discrete deformations $y
\in \Ys_0$, we first define the symbols
\begin{displaymath}
  |D^2 y(x)| := \max_{i,j = 1, \dots, 6} |D_iD_j y(x)|, \quad
  \text{and} \quad
  |D^3 y(x)| := \max_{i,j,k = 1, \dots, 6} |D_iD_jD_k y(x)|, 
  \quad \text{for } x \in \L.
\end{displaymath}
With mild abuse of notation, we then define the norms
\begin{displaymath}
  \| D^2 y \|_{\ell^p(\mathcal{A})} := \| |D^2 y| \|_{\ell^p(\mathcal{A})}, \quad \text{and}
  \quad \| D^3 y \|_{\ell^p(\mathcal{A})} := \| |D^3 y| \|_{\ell^p(\mathcal{A})},
\end{displaymath}
for any $\mathcal{A} \subset \L$ and $y \in \Ys_0$. If the label
$\mathcal{A}$ is omitted, then it is assumed that $\mathcal{A} = \L$.


\section{Construction of the GR-AC Method}
\label{sec:construstion}
In this section we carry out an explicit construction of the GR-AC
method. Our results are variants of results in \cite{E:2006}, however,
since our ansatz is different from the one used in \cite{E:2006}, and
since we wish to be precise about the equivalence of certain
conditions, we provide details for all our proofs.

We assume throughout the remainder of the paper that the reconstructed
difference $\RO_x\Da{j}y(x)$ may depend only on the original
differences $\Da{j-1}y(x), \Da{j}y(x)$, and $\Da{j+1}y(x)$, that is,
\begin{equation}
  \label{eq:simple_Cxji}
  C_{x,j,i} = 0 \quad \text{for } \b| (i-j) {\rm~mod~} 6 \b| > 1,
  \qquad i, j \in \{1,\dots,6\}, \quad x \in \Is.
\end{equation}
For future reference, we call \eqref{eq:simple_Cxji} the {\it
  one-sidedness condition}.

In \S\ref{sec:construction:e_cons} and \S\ref{sec:construction:patch}
we derive general conditions on the parameters that are independent of
the choice of the atomistic region. In \S\ref{sec:construction:flat}
and \S\ref{sec:construction:corners} we then compute explicit sets of
parameters.

\subsection{Conditions for local energy consistency}
\label{sec:construction:e_cons}
We first derive conditions for the local energy consistency condition
\eqref{eq:energy_cons}. 

\begin{proposition}
  \label{th:econs}
  Suppose that the parameters $C_{x,j,i}$ satisfy the one-sidedness
  condition~\eqref{eq:simple_Cxji}, then the interface potential
  $\Vi_x$ satisfies the local energy consistency condition
  \eqref{eq:energy_cons} for all potentials $V \in \Vs$ if and only if
  \begin{equation}
    \label{eq:equiv_econs}
    C_{x,j,j-1} = C_{x,j,j+1} = 1 - C_{x,j,j}, \qquad \text{for } j = 1, \dots, 6.
  \end{equation}
\end{proposition}
\begin{proof}
  We require that $\Vi_x(\mF{\bf a}) = V(\mF{\bf a})$, for arbitrary
  $V\in\Vs$, which is equivalent to
  \begin{displaymath}
    \mF a_j = \sum_{i = 1}^6 C_{x,j,i} \mF a_{i} \qquad \text{for } j =
    1, \dots, 6.
  \end{displaymath}
  Since this has to hold for arbitrary $\mF \in \R^{2 \times 2}$, and
  in view of \eqref{eq:simple_Cxji}, we obtain the condition
  \begin{displaymath}
    a_j = C_{x,j,j-1} a_{j-1} + C_{x,j,j} a_j + C_{x, j, j+1} a_{j+1}
  \end{displaymath}
  Since $a_j = a_{j-1} + a_{j+1}$, this is equivalent to
  \begin{displaymath}
    (C_{x,j,j-1} + C_{x,j,j} - 1) a_{j-1} +     (C_{x,j,j+1} +
    C_{x,j,j} - 1) a_{j-1}  = 0,
  \end{displaymath}
  and since $a_{j-1}, a_{j+1}$ are linearly independent, we obtain the
  condition that
  \begin{displaymath}
    C_{x,j,j-1} + C_{x,j,j} = 1, \quad \text{and} \quad
    C_{x,j,j+1} + C_{x,j,j} = 1.    
  \end{displaymath}
  Subtracting these two conditions gives $C_{x,j,j+1} = C_{x,j,j-1}$,
  and hence we obtain~\eqref{eq:equiv_econs}.
\end{proof}

As a consequence of Assumption \eqref{eq:simple_Cxji}, and Proposition
\ref{th:econs}, we have reduced the number of free parameters to six
for each site $x \in \Is$. To simplify the subsequent notation,
whenever the parameters $C_{x,j,i}$ are chosen to satisfy
\eqref{eq:equiv_econs}, we will write
\begin{equation}
  \label{eq:simple_C_econs}
  C_{x,j} := C_{x,j,j}, \quad \text{and note that } C_{x,j,j-1} =
  C_{x,j,j+1} = 1-C_{x,j}.
\end{equation}
Since it is equivalent to \eqref{eq:energy_cons} we call
\eqref{eq:simple_C_econs} the local energy consistency condition as
well.


\subsection{Conditions for local force consistency}
\label{sec:construction:patch}
We rewrite $\Eac$ in terms of a hybrid site potential
\begin{equation}
  \label{eq:defn_Vac}
  \Eac(y) = \sum_{x \in \L} \Vac_x(Dy(x)), \qquad \text{where} 
  \quad
  \Vac_x({\bf g}) :=
  \cases{
    \Vc({\bf g}),&x\in\Cs,\\
    \Vi_x({\bf g}),&x\in\Is,\\
    V({\bf g}),&x\in\As.
  }
\end{equation}

\begin{lemma}
  \label{th:fac_v1}
  Suppose that the parameters $(C_{x,j,i})_{i,j=1}^6, x \in \Is$,
  satisfy the one-sidedness condition \eqref{eq:simple_Cxji} and local
  energy consistency \eqref{eq:simple_C_econs}. Moreover, let
  \begin{equation}
    \label{eq:C_At_CB}
    C_{x,j} := 1 \quad \text{for } x \in \As \quad
    \text{and} \quad
    C_{x,j} := 2/3 \quad \text{for } x \in \Cs, \qquad 
    j = 1, \dots,6,
  \end{equation}
  and let $(C_{x,j,i})_{i,j = 1}^6$, $x \in \As \cup \Cs$, be defined
  to be compatible with \eqref{eq:simple_Cxji} and
  \eqref{eq:simple_C_econs}; then
  \begin{equation}
    \label{eq:fac_v1}
    -f^\ac(x; \mF\id) = \sum_{j=1}^6\sum_{i=1}^6 (C_{x-a_{i},j,i} -
    C_{x,j,i}) V_{\mF,j}  \qquad \forall x \in \L.
  \end{equation}
\end{lemma}
\begin{proof}
  Using the notation \eqref{eq:defn_Vac}, we have
  \begin{displaymath}
    \< \del\Eac(\mF\id), u\> = \sum_{x\in\L}\sum_{i=1}^6 
    \pp_i \Vac_x(\mF{\bf a}) \cdot \Da{i} u(x),
  \end{displaymath}
  which immediately gives
  \begin{equation}
    \label{eq:flat:facprf:1}
    - f^\ac(x; \mF\id) = \sum_{i = 1}^6 \b[ \pp_i \Vac_{x-a_i}(\mF{\bf
      a}) - \pp_i \Vac_x(\mF{\bf a}) \b].
  \end{equation}

  With the notation introduced in \eqref{eq:C_At_CB}, we obtain
  \begin{displaymath}
    \sum_{i = 1}^6 \pp_i \Vac_x(\mF{\bf a}) \cdot \Da{i} u(x)
    = \sum_{j = 1}^6 V_{\mF, j} \sum_{i = 1}^6 C_{x,j,i} \Da{i} u(x), 
  \end{displaymath}
  which implies
  \begin{equation}
    \label{eq:flat:facprf:5}
    \pp_i \Vac_x(\mF{\bf a}) = \sum_{j = 1}^6 C_{x,j,i} V_{\mF,j}.
  \end{equation}
  Combining \eqref{eq:flat:facprf:5} with \eqref{eq:flat:facprf:1}
  yields \eqref{eq:fac_v1}.
  %
  %
\end{proof}

Testing \eqref{eq:fac_v1} for all $V \in \Vs$ and $\mF \in \R^{2
  \times 2}$, we obtain the next result.

\begin{lemma}
  Suppose that the parameters $(C_{x,j,i})_{i,j = 1}^6, x \in \Is$,
  satisfy one-sidedness \eqref{eq:simple_Cxji} and local energy
  consistency \eqref{eq:equiv_econs}. Then $\Eac$ satisfies local
  force consistency~\eqref{eq:force_cons} for all $V \in \Vs$ if and
  only if
  \begin{equation}
    \label{eq:patch_cons_cond}
    \sum_{i = 1}^6 \b( C_{x-a_i, j,i} - C_{x-a_i,j+3,i}
    - C_{x, j, i} +C_{x,j+3,i} \b) = 0 
    \qquad\forall\, j = 1, 2, 3,
    \quad \forall x \in \L.
  \end{equation}
\end{lemma}
\begin{proof}
  Using \eqref{eq:fac_v1} and point symmetry \eqref{eq:pt_symm_D} one
  readily checks that \eqref{eq:patch_cons_cond} is sufficient for
  force consistency \eqref{eq:force_cons}.  To show that
  \eqref{eq:patch_cons_cond} is also necessary we test
  \eqref{eq:fac_v1} with
  \begin{displaymath}
    V({\bf g}) = \smfrac12 \b( |g_1-a_1|^2 + |g_4 - a_4|^2 \b),
  \end{displaymath}  
  which clearly belongs to the class $\Vs$, to obtain
  \begin{align*}
    -f^\ac(x; \mF\id) =~& \sum_{j = 1, 4} \sum_{i = 1}^6 (C_{x-a_i,j,i}
    - C_{x,j,i}) (\mF - \mI)a_j \\
    =~& \sum_{i = 1}^6 \b( C_{x-a_i,1,i} - C_{x-a_i,4,i} - C_{x,1,i} +
    C_{x,4,i} \b) \, (\mF - \mI) a_1.
  \end{align*}
  For this expression to vanish for all $\mF \in \R^{2 \times 2}$ we obtain
  precisely \eqref{eq:patch_cons_cond} for $j = 1$. For $j = 2,3$ the
  same argument applies.
\end{proof}

\subsection{Explicit parameters for flat interfaces}
\label{sec:construction:flat}

We now give a characterisation, for a flat a/c interface, of all
parameters satisfying the one-sidedness assumption
\eqref{eq:simple_Cxji}, which give a patch test consistent a/c
method.

\begin{figure}
  \begin{center}
   \includegraphics[height= 0.16\textwidth]{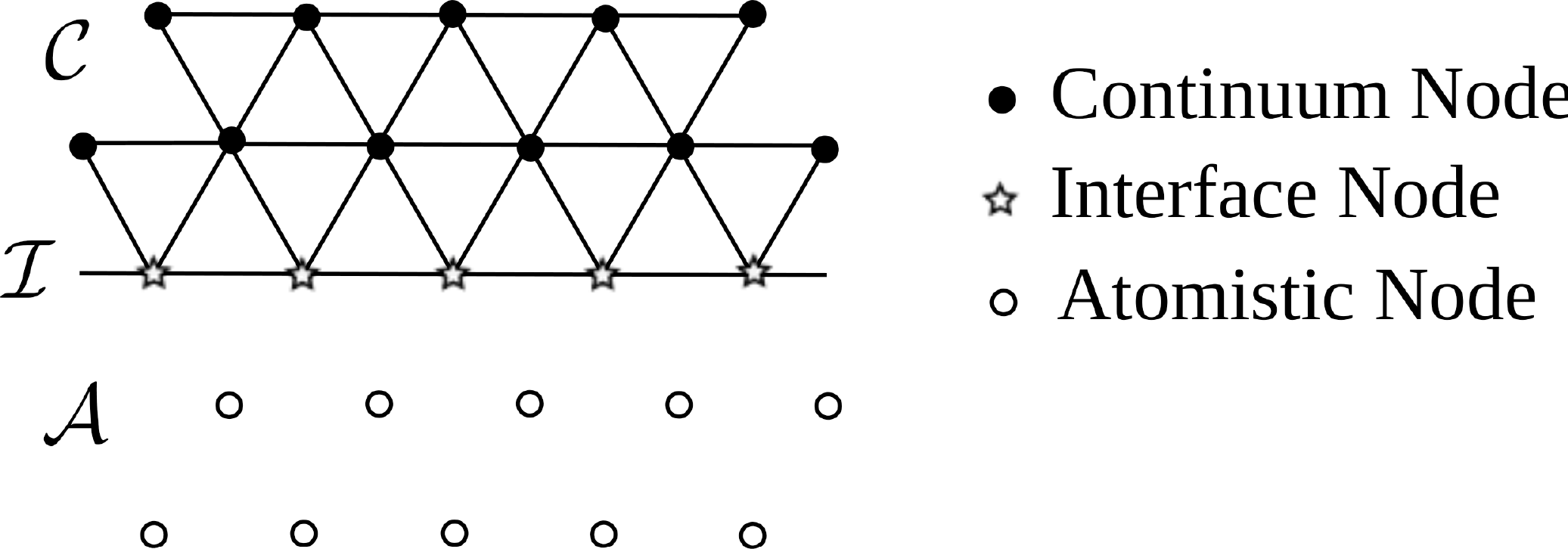}
   \end{center}
   \caption{The flat interface case.}
\label{fig:flatinterface}
\end{figure}

\begin{proposition}
  \label{th:params:flat}
  Suppose that $\As = \{ x \in \L \sep x_2 < 0 \}$, $\Is = \{x \in \L
  \sep x_2 = 0 \}$ and $\Cs = \{ x \in \L \sep x_2 > 0 \}$ (see Figure
  \ref{fig:flatinterface}). Then the parameters $(C_{x,j,i})_{i,j =
    1}^6$, $x \in \Is$, satisfy the one-sidedness condition
  \eqref{eq:simple_Cxji}, energy consistency
  \eqref{eq:simple_C_econs}, and force consistency
  \eqref{eq:patch_cons_cond}, if and only if
  \begin{align}
    \label{eq:interface_C_econs1}
    C_{x,1} =~& C_{x+a_1, 4} \qquad \forall x \in \Is, \quad
    \text{and} \\
    \label{eq:interface_C_econs2}
    C_{x,j} =~& C_{x+a_1, j} \qquad \forall x \in \Is, \quad j \in
    \{2, 3, 5, 6\},
  \end{align}
  where we have used the reduced parameters defined in
  \eqref{eq:simple_C_econs}.
\end{proposition}
\begin{proof}
  One-sidedness \eqref{eq:simple_Cxji} and energy consistency
  \eqref{eq:simple_C_econs} yields the reduced parameters
  $(C_{x,j})_{j = 1}^6$, $x \in \Is$, satisfying
  \eqref{eq:simple_C_econs}. Recall also the extension
  \eqref{eq:C_At_CB} of these parameters for $x \in \As \cup \Cs$. 

  Let $\Is_+ := \{ x+a_2 \sep x \in \Is \}$ and $\Is_- := \{ x - a_2
  \sep x \in \Is \}$. Clearly, we need to
  test~\eqref{eq:patch_cons_cond} only for $x \in \Is \cup \Is_- \cup
  \Is_+$. Exploiting the symmetries of the problem it is also clear
  that we only need to consider $j = 1, 2$.

  It is straightforward to verify through direct calculations that any
  set of coefficients satisfying \eqref{eq:interface_C_econs1},
  \eqref{eq:interface_C_econs2} satisfies the equivalent force
  consistency condition \eqref{eq:patch_cons_cond}. 
  
  Let $j=1$ and $x\in\Is$ then we obtain that
  \eqref{eq:interface_C_econs1} is necessary from the force
  consistency condition \eqref{eq:patch_cons_cond}, applied at $x+a_2$
  or $x+a_6$. Let $j=2$, then we obtain $C_{x,2}=C_{x+a_1,2}$ from the
  force consistency condition \eqref{eq:patch_cons_cond} applied at
  $x+a_2$.  Therefore, \eqref{eq:interface_C_econs1} and
  \eqref{eq:interface_C_econs2} are also necessary.
\end{proof}

\begin{remark}
  We observe that the coefficients $(C_{x,i,j})_{i,j = 1}^6, x \in
  \Is$, are not unique, but that we have considerable freedom in the
  construction of the GR-AC method: For each direction $a_i$ that is
  not aligned with the interface, there is a free parameter, while for
  each edge $(x, x+a_1)$ lying on the interface, there is one
  additional free parameter. This freedom will be reduced in the case
  of corners.
\end{remark}

\subsection{Explicit parameters for general interfaces}
\label{sec:construction:corners}
For general interface geometries we make the following separation
assumption. This assumption requires that, if the atomistic region can
be decomposed into several connected components, then they must be
separated by at least four ``lattice hops''.

\begin{assumption}
  \label{th:geninterface}
  Each vertex $x\in\Is$ has exactly two neighbours in $\Is$, and at
  least one neighbour in $\Cs$.
\end{assumption}


\medskip
As in the flat interface case, we can completely characterise all
parameters within the one-sidedness assumption, which satisfy the
patch test.

\begin{proposition}
  \label{th:params:corner}
  Let $\As \subset \L$ be defined in such a way that the interface set
  $\Is$ satisfies Assumption \ref{th:geninterface}, and is {\em not}
  planar.  Then the parameters $(C_{x,j,i})_{i,j = 1}^6$, $x \in \Is$,
  satisfy the one-sidedness condition \eqref{eq:simple_Cxji}, energy
  consistency \eqref{eq:simple_C_econs}, and force consistency
  \eqref{eq:patch_cons_cond}, if and only if
  \begin{align}
    \label{eq:interface_C_ccons1}
    C_{x,j} =~& C_{x+a_j, j+3} \hspace{-2cm}  &&\forall x \in \Is, \quad
    x+a_j\in\Is, \\
    \label{eq:interface_C_ccons2}
    C_{x,j} =~& 1  \hspace{-2cm} &&\forall x \in \Is, \quad x+a_j\in\As, \quad \text{and}\\
    \label{eq:interface_C_ccons3}
    C_{x,j} =~& 2/3 \hspace{-2cm} &&\forall x \in \Is, \quad x+a_j\in\Cs,
  \end{align}
  where $(C_{x,j})_{j = 1}^6$, $x \in \Is$, are the reduced parameters
  defined in \eqref{eq:simple_C_econs}.
\end{proposition}
\begin{proof}
  As in the flat interface case, one-sidedness \eqref{eq:simple_Cxji}
  and energy consistency \eqref{eq:simple_C_econs} are equivalent to
  having the reduced parameters $(C_{x,j})_{j = 1}^6, x \in \Is$,
  satisfying \eqref{eq:simple_C_econs}. Recall also the extension
  \eqref{eq:C_At_CB} of these parameters for $x \in \As \cup \Cs$.

  Let $\Is_+ := \{ x\in\Cs \sep \exists a_j, x+a_j \in \Is \}$ and
  $\Is_- := \{ x\in\As \sep \exists a_j, x+a_j \in \Is \}$. We need to
  test \eqref{eq:patch_cons_cond} only for $x \in \Is \cup \Is_- \cup
  \Is_+$.  The necessity of \eqref{eq:interface_C_ccons1} follows as
  in the flat interface case. The necessity of
  \eqref{eq:interface_C_ccons2} and \eqref{eq:interface_C_ccons3} can
  be obtained by testing the corner sites in $\Is_\pm$ in the
  interface geometry depicted in Figure \ref{fig:acmethod}.

  To see that
  \eqref{eq:interface_C_ccons1}--\eqref{eq:interface_C_ccons3} are
  also sufficient one notes, first, that the corresponding
  coefficients always provide zero contribution on each edge for the
  sum in \eqref{eq:patch_cons_cond}. Computing the force at $x \in
  \Is_+$ we see that the contribution from $\Vi$ is the same as from
  $\Vc$, and must therefore cancel, since the pure Cauchy--Born model
  passes \eqref{eq:patch_cons_cond}. For $x \in \Is_-$ the same
  argument applies.

  It remains to test \eqref{eq:patch_cons_cond} for $x \in \Is$, at
  corners. Since \eqref{eq:patch_cons_cond} is a local condition, and
  due to Assumption~\ref{th:geninterface}, one may assume that the
  interface has only one corner. Since all other sites are in
  equilibrium, and since the forces are conservative, it follows that
  the corner must also be in equilibrium.

  (Alternatively, one may check \eqref{eq:patch_cons_cond} through
  explicit computations for the corner geometry shown in Figure
  \ref{fig:acmethod}. All other geometries can be reduced to this one
  by symmetry.)
  %
\end{proof}

\begin{remark}
  We observe that, for a general interface, we only have freedom to
  choose the geometric reconstruction parameters along the interface,
  namely, for each interface edge there is one free parameter.
\end{remark}


\section{Consistency of the Cauchy--Born Approximation}
Before we embark on the analysis of the GR-AC method
\eqref{eq:Eac_general}, we establish a sharp consistency estimate for
Cauchy--Born approximation. Related results were established in
\cite{E:2007a}, which require more stringent conditions on the
smoothness of the deformation field. For the analysis of a/c methods a
sharp consistency estimate, such as Theorem \ref{th:cons2}, is
useful. In the remainder of the section we establish technical results
that are useful for the subsequent consistency analysis of the GR-AC
method.

\subsection{Second-order consistency}
\label{sec:cb_o2}
A natural way to represent the first variation of $\Ea$ is
\begin{equation}
  \label{eq:delEa:edges}
  \b\< \del\Ea(y), u \b\> 
  = \sum_{x \in \L} \sum_{j = 1}^6 \pp_j V(\Da{}y(x)) \cdot \Da{j} u(x) \\
  = \sum_{x \in \L} \sum_{j = 1}^6 V_{x,j} \cdot \Da{j} u(x),
\end{equation}
where we use the notation $V_{x,j} := \pp_j V(\Da{}y(x))$. This
representation can be interpreted as a sum over mesh edges. By
contrast, the most natural representation of $\del\Ec$ is
\begin{equation}
  \label{eq:delEc:volume}
  \b\< \del\Ec(y), u \b\> = \sum_{T \in \T} |T| \pp W(\Dc{T} y) : \Dc{T} u.
\end{equation}
To estimate $\del\Ea - \del\Ec$ we will rewrite
\eqref{eq:delEc:volume} in a form mimicking
\eqref{eq:delEa:edges}. The opposite approach is also possible, but
does not lead as easily to second-order consistency estimates.

\begin{lemma}
  \label{th:delEcdelEa:edge_1}
  For $y \in \Ys_0, T \in \T$, let $V_{T, j} := \pp_j V(\Dc{T} y\cdot  {\bf
    a})$; then
  \begin{align}
    \label{eq:delEc:edge_2}
    \b\< \del\Ec(y), u \b\> =~& \sum_{x \in \L} \sum_{j = 1}^3
    \b(V_{T_{x,j}, j} + V_{T_{x,j-1},j} \b) \cdot \Da{j} u(x), \quad
    \forall u \in \Us_0,  \quad
    \text{and} \\[-2mm]
    \label{eq:delEa:edge_2}
    \b\< \del\Ea(y), u \b\> =~& \sum_{x \in \L} \sum_{j = 1}^3
    \b(V_{x, j} - V_{x+a_j,j+3}\b) \cdot \Da{j} u(x), \quad
    \forall u \in \Us_0.
  \end{align}
\end{lemma}
\begin{proof}
  It is easy to see that
  \begin{equation}
    \label{eq:DW_sumj}
    \pp W(\mF) = \frac{1}{\Omega_0} \sum_{j = 1}^6 \pp_j V(\mF {\bf a}) \otimes a_j,
  \end{equation}
  and hence, using $\Omega_0 = 2 |T|$ and $\Dc{T} u \cdot a_j= \Da{j} u(x_{T,j})$,
  \begin{align*}
    \b\< \del\Ec(y), u \b\> = \frac{1}{\Omega_0}\sum_{T \in \T} |T| \sum_{j = 1}^6 \b[V_{T, j}
    \otimes a_j \b] : \Dc{T} u 
    = \frac{1}{2} \sum_{T \in \T} \sum_{j = 1}^6 V_{T, j} \cdot \Da{j} u(x_{T, j}).
  \end{align*}
  Every edge appears twice in this sum since it is shared between two
  elements; hence we obtain the edge representation
  \begin{equation}
    \label{eq:delEc:edge_1}
    \b\< \del\Ec(y), u \b\> = \sum_{x \in \L} \sum_{j = 1}^6
    \smfrac12 \b(V_{T_{x,j}, j} + V_{T_{x,j-1},j} \b) \cdot \Da{j} u(x) \qquad
      \forall u \in \Us_0.
  \end{equation}

  Since $\Da{j+3} u(x+a_j) = - \Da{j} u(x)$, and using $V_{T,j+3} =
  -V_{T,j}$ (see \eqref{eq:pt_symm_D}) we can reduce this sum as
  follows:
  \begin{align*}
    \b\< \del\Ec(y), u \b\> =~& \sum_{x \in \L} \sum_{j = 1}^3
    \smfrac12 \b(V_{T_{x,j}, j} + V_{T_{x,j-1},j} -
    V_{T_{x,j},j+3} - V_{T_{x,j-1},j+3} \b) \cdot \Da{j} u(x) \\
    =~& \sum_{x \in \L} \sum_{j = 1}^3
    \b(V_{T_{x,j}, j} + V_{T_{x,j-1},j} \b) \cdot \Da{j} u(x).
  \end{align*}
  This concludes the proof of (\ref{eq:delEc:edge_2}). 

  For the proof of (\ref{eq:delEa:edge_2}) one only needs to use the
  identity $\Da{j+3} u(x+a_j) = - \Da{j} u(x)$.
\end{proof}

\begin{theorem}
  \label{th:cons2}
  Let $y \in \Ys_0$, then
  \begin{equation}
    \label{eq:cons2}
    \b\| \del\Ea(y) - \del\Ec(y) \b\|_{\Us^{-1,p}} \leq
    c \b(M_2 \| D^3 y
    \|_{\ell^p} + M_3 \| D^2 y
    \|_{\ell^{2p}}^2\b)
  \end{equation}
  where $M_2, M_3$ are defined in \S\ref{sec:prop_V}.
\end{theorem}
\begin{proof}
  It is useful to visualize this proof using Figure
  \ref{fig:cbconsistency}, and Figure \ref{fig:Txj_xTj} for additional
  detail. From Lemma \ref{th:delEcdelEa:edge_1} we obtain
  \begin{align}
    \b\< \del\Ea(y) - \del\Ec(y), u \b\> =~& \sum_{x \in \L} \sum_{j =
      1}^3 \delta_j(x) \cdot \Da{j} u(x), \\
    \label{eq:cons2_defn_epsj}
    \text{where} \qquad 
   \delta_j(x) :=~& V_{x, j} - V_{x+a_j,j+3} - V_{T_{x,j}, j} -
    V_{T_{x,j-1},j}.
  \end{align}
  In the following we estimate $\delta_1(x)$ only; the remaining
  estimates follow by symmetry.

  \begin{figure}
    \begin{center}
      \includegraphics[height= 0.15\textwidth]{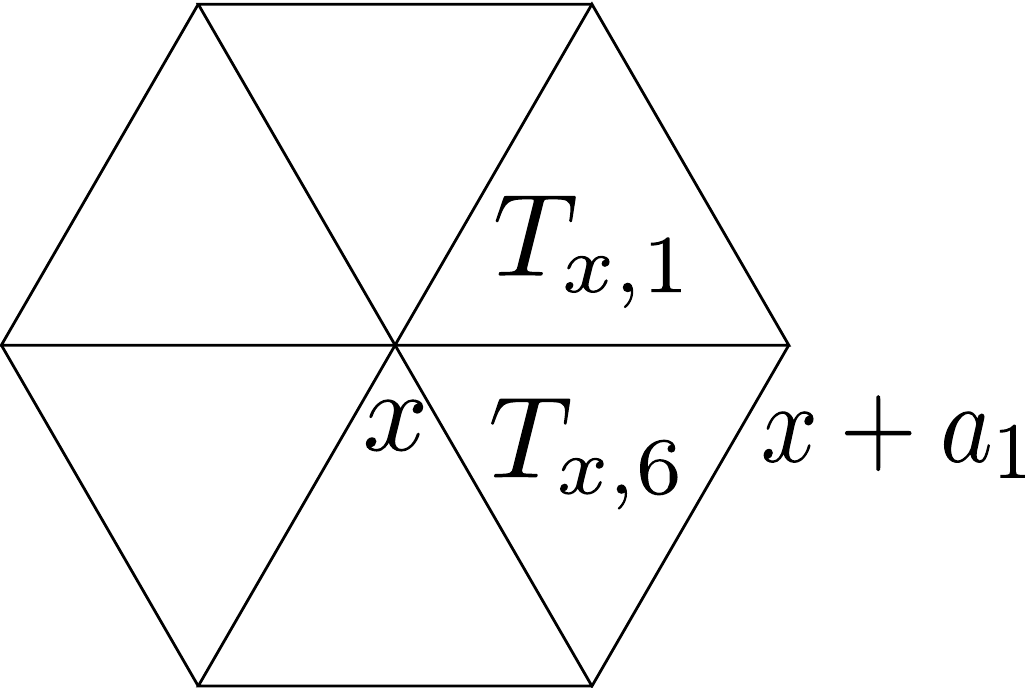}
    \end{center}
    \caption{Visualisation of the proof of Theorem \ref{th:cons2}.}
    \label{fig:cbconsistency}
  \end{figure}

  Let $\mF_+ := \Dc{T_{x,1}} y$ and $\mF_- := \Dc{T_{x,6}}y$, then
  $V_{T_{x,1},1} = V_{\mF_+, 1}$ and $V_{T_{x,6},1} =
  V_{\mF_-,1}$. Moreover we can Taylor expand
  \begin{align*}
    V_{x,1} =~& V_{\mF_+,1} + \sum_{i = 1}^6 V_{\mF_+,1i} (\Da{i}y(x) -
    \mF_+ a_i) + O\b(|D^2y(x)|^2\b), \quad \text{and similarly} \\
    - V_{x+a_1,4} =~& - V_{\mF_-,4} - \sum_{j = 1}^6 V_{\mF_-, 4i}
    (\Da{i} y(x+a_1) - \mF_- a_i)  + O\b(|D^2y(x)|^2\b)\\
    =~& V_{\mF_-, 1} - \sum_{j = 1}^6 V_{\mF_-,1(i+3)} (\Da{i}
    y(x+a_1) - \mF_- a_i)  + O\b(|D^2y(x)|^2\b)\\
    =~& V_{\mF_-, 1} + \sum_{j = 1}^6 V_{\mF_-,1i} (- \Da{i+3}
    y(x+a_1) - \mF_- a_i) + O\b(|D^2y(x)|^2\b).
  \end{align*}
  A careful analysis of the remainder shows that $O(|D^2y(x)|^2) \leq
  \frac{1}{2}\sum^6_{i,j=1} |\pp_{1ij}V({\bm \theta})|\,|D^2y(x)|^2|$
  for some ${\bm \theta} \in \R^{2\times 6}$.  In the remainder of the
  proof we will suppress the argument ${\bm \theta}$.


  Clearly, $V_{\mF_-,1i} - V_{\mF_+,1i} = O(|D^2y(x)|) \leq
  \sum_{j=1}^6 |\pp_{1ij} V|\,|D^2y(x)|$,
  and hence we can deduce that
  \begin{align*}
    \delta_1(x) =~& \sum_{i = 1}^6 V_{\mF_+,1i} \b( \Da{i}y(x) - \mF_+
    a_i - \Da{i+3} y(x+a_1) - \mF_- a_i \b) + O\b(|D^2y(x)|^2\b) \\
    =~& \sum_{i = 1}^6 V_{\mF_+,1i} \b( \Da{i}y(x) - \Da{i} y(x_i^+) +
    \Da{i} y(x+a_1-a_i) - \Da{i} y(x_i^-) \b) + O\b(|D^2y(x)|^2\b) \\
    =:~& \sum_{i = 1}^6 V_{\mF_+,1i} \, \eps_i + O\b(|D^2y(x)|^2\b),
  \end{align*}
  where $x_i^+ := x_{T_{x,1}, i}$ and $x_i^- := x_{T_{x,6},
    i}$. (These are simply the vertices in the two adjacent elements
  such that the identities $\mF_\pm a_i = \Da{i} y(x_i^\pm)$ hold.)

  Tracing the previous Taylor expansions, we see that, in the last
  estimate, $O(|D^2y(x)|^2) \leq
  2\sum^6_{i,j=1}|\pp_{1ij}V|\,|D^2y(x)|^2$.

  We compute $\eps_3$ in detail but only give the results for the
  remaining coefficients:
  \begin{align*}
    \eps_3 =~& \Da{3}y(x) - \Da{3} y(x+a_1) +
    \Da{3} y(x+a_1-a_3) - \Da{3} y(x-a_3)  \\
    =~& - \Da{1} \Da{3} y(x) + \Da{1}\Da{3} y(x-a_3) = \Da{6}\Da{1}\Da{3} y(x).
  \end{align*}
  By performing similar calculations for $i = 1, 2, 4, 5, 6$, one
  finds
  \begin{align*}
    \eps_1 = \eps_2 = \eps_6 = 0, \quad 
    \eps_4 = \Da{1}\Da{1}\Da{4}y(x), \quad \text{and} \quad
    \eps_5 = \Da{1}\Da{2}\Da{5} y(x);
  \end{align*}
  hence we obtain that $\delta_j(x) = O(|D^2 y(x)|^2 + |D^3 y(x)|)$
  (recall that we assumed, without loss of generality, that $j = 1$),
  where $O(|D^3 y(x)|) \leq \sum_{i=3,4,5}|\pp_{1,i} V||D^3y(x)|$.

  Combining these estimates, we obtain
  \begin{align*}
    \b\< \del\Ea(y) - \del\Ec(y), u \b\> 
    \leq~& \B(\sum_{x \in \L} \sum_{j = 1}^3 |\delta_j(x)|^p
    \B)^{1/p}\,
    \B( \sum_{x \in \L} \sum_{j = 1}^3 |\Da{j} u(x)|^{p'} \B)^{1/p'}.
  \end{align*}
  Elementary estimates yield
  \begin{align*}
     \B(\sum_{x \in \L} \sum_{j = 1}^3 |\delta_j(x)|^p
    \B)^{1/p} \leq~& M_2 \|D^3 y \|_{\ell^p} + M_3 \|D^2 y
    \|_{\ell^{2p}}^2, \quad \text{and} \\
    \B( \sum_{x \in \L} \sum_{j = 1}^3 |\Da{j} u(x)|^{p'} \B)^{1/p'}
    \leq~& \b(2\sqrt{3}\b)^{1/p'} \B( \sum_{T \in \T} |T| \b|\Dc{T}u\b|^{p'}\B)^{1/p'},
  \end{align*}
  from which the result follows immediately.
\end{proof}

In the following subsections, we derive technical results related to
Theorem \ref{th:cons2}, in preparation for the proof of consistency of
the GR-AC method.

\subsection{Stress tensors}
\label{eq:cb_stress}
We can re-interpret Theorem \ref{th:cons2} in terms of a second-order
error estimate for certain stress tensors. If, for some $y \in \Ys_0$,
there exist tensor fields $\Sa(y; \bullet), \Sc(y; \bullet) \in
\PO(\T)^{2\times 2}$, which satisfy the identities
\begin{align}
  \<\del\Ea(y),u\>&=\sum_{T\in\T}|T|\Sa(y;T):\Dc{T} u, \quad
  \text{and} 
  \label{eq:atomstress}\\
  \<\del\Ec(y),u\>&=\sum_{T\in\T}|T|\Sc(y;T):\Dc{T} u
  \label{eq:contstress}
\end{align}
then we call $\Sa$ an atomistic stress tensor and $\Sc$ a continuum
stress tensors.

It follows from \eqref{eq:delEa:edges} and \eqref{eq:delEc:volume} that
\begin{align}
    \label{eq:defn_Sa}
    \Sa(y; T) :=~& \frac{1}{\Omega_0} \sum_{j = 1}^6 V_{x_{T,j},j} \otimes
    a_j, \qquad \text{and} \\
    \label{eq:defn_Sc}
    \Sc^1(y; T) :=~& \pp W(\Dc{T} y)=\frac{1}{\Omega_0} \sum_{j = 1}^6 V_{T,j} \otimes a_j
\end{align}
satisfy \eqref{eq:atomstress} and \eqref{eq:contstress},
respectively. As we will see immediately, they are not the unique
choices.

In the following calculation (and later on as well) we denote by $T_j$
the unique neighbouring element of $T \in \T$, which shares an edge
with direction $a_j$ with $T$; see Figure~\ref{fig:neighbortri}.

\begin{figure}
  \begin{center}
    \includegraphics[height= 2.5cm]{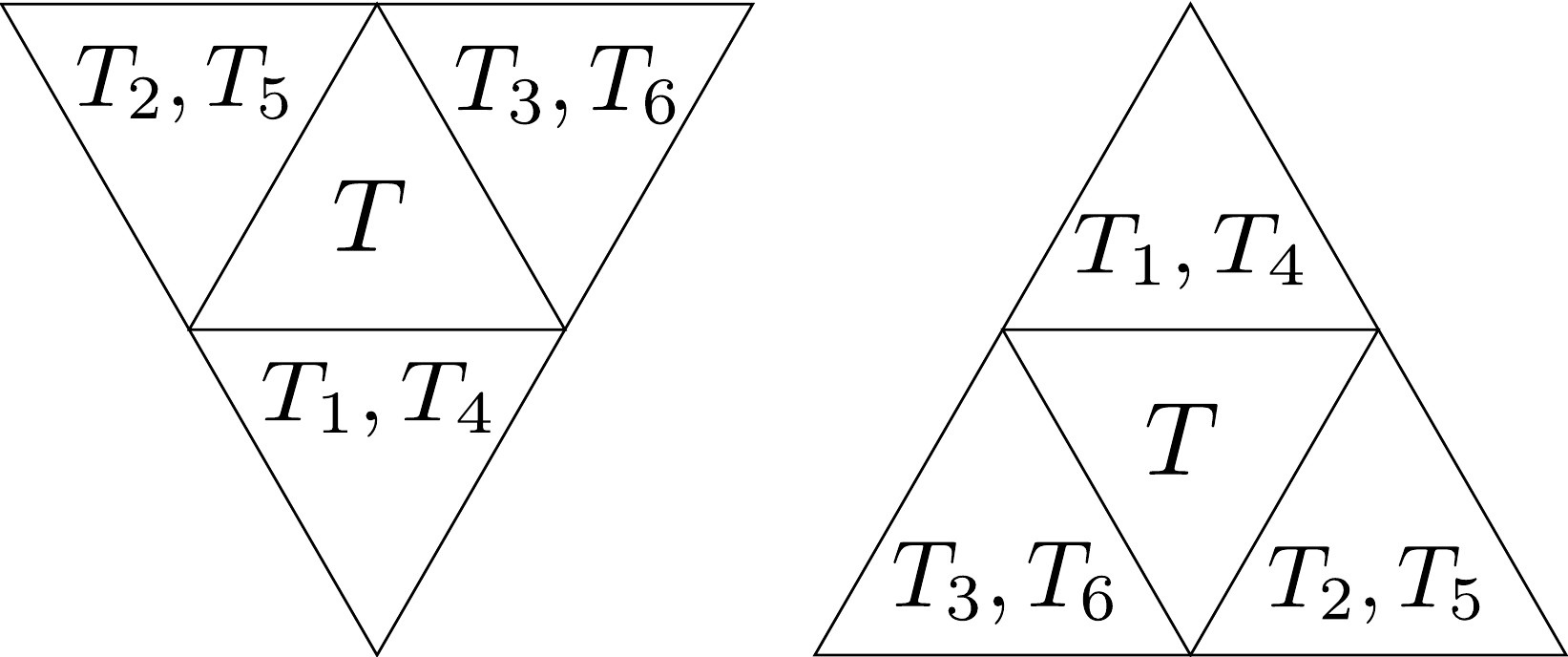} 
  \end{center}
  \caption{Notation for neighbouring triangles of $T\in\T$.}
  \label{fig:neighbortri}
\end{figure}

With this notation, and using the fact that $\Da{j}u(x_{T,j}) =
\Da{j}u(x_{T_j,j})$, we observe that
\begin{align}
  \label{eq:construction_Sc2}
 \b\< \del\Ec(y), u \b\> =~&  \sum_{T \in \T}|T| 
 \frac{1}{\Omega_0} \sum_{j = 1}^6 V_{T, j} \cdot \Da{j} u(x_{T, j})\\
 \nonumber
 = &  \sum_{T \in \T} |T| \frac{1}{\Omega_0} 
 \sum_{j = 1}^6 \frac{1}{2} \b(V_{T,j} + V_{T_{j},j} \b) \cdot \Da{j}
 u(x_{T,j})\\ 
 \nonumber
 =~&  \sum_{T \in \T} |T|
 \bg\{ \frac{1}{\Omega_0}\sum_{j = 1}^6 \frac{1}{2}
                         \b(V_{T,j} + V_{T_{j},j} \b) \otimes a_j \bg\}
                         : \D_T u \qquad
			      \forall u \in \Us_0,
\end{align}
which yields the alternative continuum stress tensor
\begin{equation}
  \label{eq:defn_Sc2}
  \Sc^2(y; T) := \frac{1}{\Omega_0}\sum_{j = 1}^6 \frac{1}{2} \b(V_{T,j} + V_{T_{j},j} \b)\otimes a_j.
\end{equation}

Furthermore, if we write the Cauchy--Born energy in terms of the site
energy \eqref{eq:defn_Vc}, and apply the procedure used to derive
$\Sa$, then we obtain a third variant of the continuum stress tensor:
\begin{equation}
  \label{eq:defn_Sc3}
  \Sc^3(y; T) :=\frac{1}{\Omega_0} \sum_{j = 1}^6 \Vc_{x_{T,j},j} \otimes a_j.
\end{equation}

We see that stress tensors are not uniquely defined by
\eqref{eq:contstress} and \eqref{eq:atomstress}. This causes
analytical difficulties when deriving consistency error estimates,
which strongly depend on the choice of the stress tensors. For example
we will show in the following result that $\Sc^2$ is second-order
consistent. By contrast, $\Sc^1$ and $\Sc^3$ are only first-order
consistent (cf. Remark \ref{rem:scb3_o1}).

\begin{lemma}
  \label{th:cb_stress}
  Let $y \in \Ys_0$, then 
  \begin{equation}
    \label{eq:stress_cb_o2}
    \b| \Sa(y; T) - \Sc^2(y; T) \b| \leq 
    c \b(M_3|D^2 y(x)|^2 + M_2|D^3 y(x)|\b)    
  \end{equation}
  for all $T \in \T, x \in T$.
\end{lemma}
\begin{proof}
  This estimate is obtained by reversing the construction of $\Sc^2$
  in \eqref{eq:construction_Sc2}, and applying the estimates obtained
  in the proof of Theorem \ref{th:cons2}.
\end{proof}

\begin{remark}
  \label{rem:scb3_o1}
  Taylor expansions show that $\Sc^k$, $k = 1,3$, are only first-order
  consistent, 
  \begin{displaymath}
    \b| \Sa(y; T) - \Sc^k(y; T) \b| \leq c M_2|D^2
    y(x)| \quad \text{for } x \in T,
 \end{displaymath}
 but that a second-order estimate such as \eqref{eq:stress_cb_o2}
 would be false. The first-order estimate can also be obtained from
 the fact that $\Sa(\yF; \bullet) = \Sc^k(\yF; \bullet) = \D W(\mF)$
 for all $\mF\in \R^{2 \times 2}$.
\end{remark}

\subsection{Divergence-free stress tensors}

In the previous subsection, we have seen that the stress functions
defined in \eqref{eq:atomstress} and \eqref{eq:contstress} are not
unique. It is therefore crucial to characterize all divergence-free
tensors, which is the purpose of the present section. We call a
piecewise constant tensor $\sigma\in\PO(\T)^{2\times 2}$
\textit{divergence free}, if it satisfies 
\begin{equation}
  \label{eq:kerdu}
  \int_{\R^2} \sigma : \Dc{} u \dx = \sum_{T\in\T} |T| \sigma(T) : \D_T u = 0 \qquad \forall u \in \Us_c.
\end{equation}

Divergence-free tensors can be characterised as 2D-curls of
non-conforming Crouzeix--Raviart finite elements. Let $\CR(\T)$ be
defined by
\begin{displaymath}
  \CR(\T):=\Big\{v:\R^2\to\R\,\Big{|}
  \substack{\mbox{$v|_{\text{int}(T)}$ is linear for each $T\in\T$} \\
    \mbox{ $v$ is continuous at all edge midpoints }}\Big\}.
\end{displaymath}
The degrees of freedom for functions $w \in \CR(\T)$ are the nodal
values at edge midpoints, $w(q_f)$, $f \in \Fs$, and the associated
nodal basis functions are denoted by $\zeta_f$.

We have the following characterization lemma
\cite{PolthierPreuss:2002} for divergence free tensor fields.
Although we will never use the equivalence of the characterisation
explicitly, it motivates much of our subsequent analysis. ***!***

\begin{lemma}
  \label{th:divfree}
  A tensor field $\sigma \in \PO(\T)^{2\times 2}$ is divergence-free
  (i.e., satisfies \eqref{eq:kerdu}) if and only if there exists
  $\psi\in \CR(\T)^2$, such that $\sigma = \D \psi \mJ$, where $\mJ$
  is the rotation by $\pi/2$.
\end{lemma}
\begin{proof}
  It is easy to show that every tensor of the form $\sigma = \D w
  \mJ$, $w \in \CR(\T)^2$ satisfies \eqref{eq:kerdu}, by checking the
  result for a single nodal basis function $\psi = \zeta_f$.

  To show the reverse, let $\Omega$ be a simply connected domain,
  which is a union of triangles $T \in \T$. Suppose that the number of
  vertices in $\Omega$ is $\#V$, the number of interior vertices is
  $\#V_I$, the number of edges in $\Omega$ is $\#E$, and the number of
  triangles in $\Omega$ is $\#T$.

  We test \eqref{eq:kerdu} for all $u \in \Us_0$ that are non-zero
  only in the interior of $\Omega$. The dimension of all $\sigma \in
  \PO(\Omega)^{2 \times 2}$ satisfying \eqref{eq:kerdu} for those $u$
  can be at most $4\#T-2\#V_I$. On the other hand, the dimension of
  $\CR(\Omega)^2$ is $2\# E$ and the dimension of rotated gradients of
  Crouzeix--Raviart functions, denoted by $\D \CR(\Omega)^2 \mJ$, is
  $2\# E - 2$. We will show below that the following formula holds:
  \begin{equation}
    \label{eq:mesh_counting}
    4\#T-2\#V_I \leq 2 \# - 2,
  \end{equation}
  which immediately implies that the subspace of divergence-free
  tensor coincides with $\D \CR(\Omega)^2 \mJ$. Moreover, the
  representation is of course unique (up to a shift) and therefore
  independent of the choice of the domain.

  To prove \eqref{eq:mesh_counting}, we use Euler's formula,
  \begin{equation}
    \label{eq:mesh_counting_1}
    \#V-\#E+\#T = 1,
  \end{equation}
  and the identify
  \begin{equation}
    \label{eq:mesh_counting_2}
     3\#F=2\#E-\#V+\#V_I,
  \end{equation}
  which is obtained by a simple counting argument. (Note that $\#V -
  \#V_I$ is the number of boundary edges.)  Subtracting
  \eqref{eq:mesh_counting_1} from \eqref{eq:mesh_counting_2} yields
  \eqref{eq:mesh_counting}.
\end{proof}

\subsection{Continuum stress tensor correctors}
We have different forms of continuum stress $\Sc^1$, $\Sc^2$ and
$\Sc^3$, which all can be used to represent $\del\Ec$ in the form
\eqref{eq:contstress}, and hence their differences must be divergence
free.  Lemma \ref{th:divfree} characterises the form of these
differences and motivates the following result.

\def\psic{\psi^{23}}
\begin{lemma}
  \label{th:psic}
  Let $y \in \Ys_0$, then there exists a corrector $\psic(y; \bullet)
  \in \CR(\T)^2$ satisfying the following two properties:
  \begin{align}
    \label{eq:psic_correctorprop}
    \text{Corrector property: } ~& &
    \Sc^3(y;T) - \Sc^2(y;T) =~& \D \psic(y; T) \mJ
    \qquad \forall\, T \in \T; \\
    \text{Lipschitz property: } ~& &
    \label{eq:psic_lipprop}
    \b| \psic(y; m_f)\b| \leq~& \smfrac16 M_2 \| D^2 y \|_{\ell^\infty(f
      \cap \L)} \qquad \forall f \in \Fs.
  \end{align}
\end{lemma}
\begin{proof}
  Property~\eqref{eq:psic_correctorprop} follows of course from Lemma
  \ref{th:divfree}, however, to establish~\eqref{eq:psic_lipprop} we
  require an explicit expression of $\psic$. We give the details of
  the proof for the case of an upward pointing triangle $T \in \T$
  (cf. the left configuration in Figure \ref{fig:neighbortri}). An
  elementary computation, starting from \eqref{eq:defn_Sc2} and
  \eqref{eq:defn_Sc3} and using the symmetry property
  \eqref{eq:pt_symm_D}, yields
  \begin{align*}
    \label{eq:diffcstress3}
     \Sc^3(y; T) - \Sc^2(y;T) =~& \frac{1}{3\Omega_0} \b[(V_{T,1} -
     V_{T_1,1}) + (V_{T,3} - V_{T_1,3}) + (V_{T,5} - V_{T_1,5})\b]
     \otimes a_1 + \dots,
  \end{align*}
  where ``$\dots$'' stands for terms that are symmetric to the ones in
  the first line. The directions $a_1, a_3, a_5$ are chosen
  anti-clockwise with respect to the element $T$.

  We now observe that, if $f$ is an edge of $T$ with direction $a_j$,
  $j \in \{1,3,5\}$, then 
  \begin{equation}
    \label{eq:DzetaJ}
    \D\zeta_f \mJ = \cases{
      - \smfrac{2}{\Omega_0} a_j^\top, & \text{ in } T, \\[1mm]
      \smfrac{2}{\Omega_0} a_j^\top, & \text{ in } T_j.
    }
  \end{equation}
  Let $f$ be the edge of $T$ with direction $a_1$, then choosing
 \begin{equation}
    \label{eq:defn_psic}
    \psic(y; m_f) := \frac{1}{6}(V_{T,1} - V_{T_1,1}) + \frac{1}{6}(V_{T,3} - V_{T_1,3}) + \frac{1}{6}(V_{T,5} - V_{T_1,5}),
  \end{equation}
  and making analogous choices for the remaining edges, we obtain
  \eqref{eq:psic_correctorprop}.

  With this explicit representation we can now prove the Lipschitz property
  \eqref{eq:psic_lipprop}. Let $f$ denote the edge of $T$ with
  direction $a_1$, $\mF := \Dc{T} y$ and $\mF_1 := \Dc{T_1} y$; then 
  \begin{align}
    \notag
    \b| \psic(y; m_f) \b| \leq~& \smfrac16 \b| V_{\mF_1,1} -
    V_{\mF,1}\b| + \smfrac16 \b| V_{\mF_1,2} - V_{\mF,2} \b| + \smfrac16 \b| V_{\mF_1,3} - V_{\mF,3} \b|\\
    \label{eq:psic_lipproof:1}
    \leq~& \smfrac16 \sum_{j = 1}^6 \b(| V_{,1j}|+|V_{,2j}|+|V_{,3j}|\b) \b|
    (\mF_1 - \mF) a_j \b| 
    \leq \smfrac16 M_2 \max_{j = 1, \dots, 6} \b|(\mF_1-\mF)a_j\b|,
  \end{align}
  where $V_{,1j} = \pp_{1j} V(\mG_j \cdot {\bf a})$ for some $\mG_j \in
  \R^{2 \times 2}$, and $M_2$ is defined in \S\ref{sec:prop_V}.  One
  now verifies that
  \begin{displaymath}
    (\mF_1 - \mF) a_1 = 0,
    \quad
    (\mF_1 - \mF) a_2 = \Da{6}\Da{2} y(x_{T,1}), 
    \quad \text{and} \quad
    (\mF_1 - \mF) a_3 = \Da{5}\Da{3} y(x_{T,5}),
  \end{displaymath}
  which implies
  \begin{displaymath}
    \max_{j = 1, \dots, 6} \b|(\mF_1-\mF)a_j\b| \leq
    \max\b(|D^2y(x_{T,1})|, |D^2y(x_{T,4})|\b).
  \end{displaymath}
  Combining this estimate with \eqref{eq:psic_lipproof:1} we obtain
  \eqref{eq:psic_lipprop} for edges aligned with $a_1$. The remaining
  cases follow from symmetry considerations.
\end{proof}


\section{Consistency of the GR-AC Method}
\label{sec:consistency}
We are now ready to state the second main result of this paper. The
proof is established in \S\ref{sec:Sac} through
\S\ref{sec:main_proof}. For the remainder of this section we assume
that the hypotheses stated in Theorem~\ref{th:2d:stress_err_Is} hold.

\def\Isext{\Is^{\rm ext}}

\begin{theorem}
  \label{th:2d:stress_err_Is}
  Let $\Eac$ be defined by \eqref{eq:defn_Vac}, with parameters
  $(C_{x,i,j})_{i,j = 1}^6$, $x \in \Is$, satisfying the one-sidedness
  condition \eqref{eq:simple_Cxji}, as well as the patch test
  conditions \eqref{eq:energy_cons} and \eqref{eq:force_cons}. Suppose
  in addition that the parameters are bounded, that is, 
  \begin{displaymath}
    \sup_{x \in \Is} \max_{j, i \in \{1,\dots,6\}}
    |C_{x,j,i}| =: \bar{C} < +\infty.
  \end{displaymath}
  Then there exists a constant $C_\Is = C_\Is(\bar{C})$, such that
  \begin{equation}
    \label{eq:main_cons_est}
    \b\| \del\Eac(y) - \del\Ea(y) \b\|_{\Us^{-1,p}} \leq c \b(C_\Is M_2 \|
    D^2 y\|_{\ell^{p}(\Isext)}  + M_2 \| D^3
    y\|_{\ell^p(\Cs)} + M_3 \| D^2 y\|_{\ell^{2p}(\Cs)}^2 \b),
  \end{equation}
  where $\Isext := \{ x \in \L \bsep {\rm dist}(x, \Is) \leq 1 \}$ is an
  extended interface region.
\end{theorem}


\subsection{An a/c stress tensor}
\label{sec:Sac}
Following the construction of $\Sa$ in \eqref{eq:defn_Sa} (with $\Ea$
replaced by $\Eac$), we obtain a representation of $\del\Eac$ in terms
of an a/c stress $\Sac$: let $y \in \Ys_0$ and $u \in \Us_0$, then 
\begin{align}
  \label{eq:Sac_represenation}
  \b\< \del\Eac(y), u \b\> =~& \sum_{T \in \T} |T| \Sac(y; T) : \Dc{T}
  u, \qquad \text{where} \\
  \label{eq:defn_Sac}
  \Sac(y; T) :=~& \frac{1}{\Omega_0} \sum_{j = 1}^6 \Vac_{x_{T,j},j}\otimes a_j,
\end{align}
and we recall that $\Vac_{x,j} = \pp_j\Vac(x; Dy(x))$. We now require
the following additional notation:
\begin{align}
  \notag
  \Ta :=~& \{T\in\T \sep T\cap(\Is\cup\Cs)=\emptyset\}, & 
  \Fsa :=~& \Fs \cap \Ta, \\
  \label{eq:defn_Taci_Faci}
  \Tc :=~& \{T\in\T \sep T\cap(\Is\cup\As)=\emptyset\},  &
  \Fsc :=~& \Fs \cap \Tc, \\
  \notag
  \Ti :=~& \T \setminus (\T_\Cs\cup\T_\As), \quad \text{and} &
  \Fsi :=~& \Fs \setminus (\Fsc\cup\Fsa).
\end{align}

\def\psiac{\psi^\ac}
\def\psiy{\hat{\psi}^\ac}
\begin{lemma}
  \label{th:ac_stress}
  {\it (i) } Let $\Sac$ be defined by \eqref{eq:defn_Sac}, then, for
  all $y \in \Ys_0$,
  \begin{align}
    \label{eq:Sac_eq_sa}
    \Sac(y;T) =~& \Sa(y;T) \qquad \forall T \in \Ta, \quad \text{and}
    \\
    \label{eq:Sac_eq_Sc3}
    \Sac(y;T) =~& \Sc^3(y;T) \qquad \forall T \in \Tc.
  \end{align}
 
  {\it (ii) } Let $\mF \in \R^{2 \times 2}$; then there exists a
  unique $\psiac(\mF; \bullet) \in \CR(\T)^2$ such that
  \begin{align}
    \label{eq:defn_psi_1}
    \Sac(\yF; T) - \Sa(\yF; T) =~& \D \psiac(\mF; T) \mJ \qquad \forall T
    \in \T, \quad \text{and} \\
    \label{eq:defn_psi_2}
    \psiac(\mF; m_f) =~& 0 \qquad \forall f \in \Fsa \cup \Fsc.
  \end{align}
  Moreover, there exists $L_\ac$ depending only on $\bar{C}$ such that
  the following Lipschitz property holds:
  \begin{equation}
    \label{eq:psi_Lip}
    \b|\psiac(\mF;m_f)-\psiac(\mG;m_f)\b| 
    \leq L_\ac M_2 |\mF-\mG|  
   \qquad \forall \mF, \mG \in \R^{2 \times 2}, \quad f \in \Fsi.
  \end{equation}  
\end{lemma}


\begin{proof}
  {\it (i) } Properties \eqref{eq:Sac_eq_sa} and \eqref{eq:Sac_eq_Sc3}
  follow immediately from the definitions of the three tensors and the
  sets $\Ta$ and $\Tc$, and are independent of the choice of the
  reconstruction parameters at the interface.

  {\it (ii) } Since $\Eac$ is assumed to satisfy local force
  consistency \eqref{eq:force_cons}, we have
  \begin{displaymath}
    0 = \b\< \del\Eac(y_\mF)-\del\Ea(y_\mF), u \b\>
    = \sum_{T\in\T} |T| (\Sac(y_\mF;T)-\Sa(y_\mF;T)):\Dc{T} u \qquad
    \forall u \in \Us_0,
  \end{displaymath}
  and hence $\Sac(y_\mF;\bullet)-\Sa(y_\mF; \bullet)$ is divergence
  free. According to Lemma \ref{th:divfree} there exists a function
  $\psiac \in \CR(\T)^2$, which is unique up to a constant shift, such
  that \eqref{eq:defn_psi_1} holds. Property \eqref{eq:defn_psi_2}
  uniquely determines the shift. 

  As a matter of fact, it is highly non-trivial whether
  \eqref{eq:defn_psi_2} can be satisfied, and it is in principle
  possible that the corrections ``propagate'' into the continuum
  region \cite{Ortner:2011:patch}. We postpone the detailed
  computations required to prove this to Appendix
  \ref{sec:flat:stress} and \ref{sec:corner:stress}, where we then
  also give a proof of the Lipschitz property~\eqref{eq:psi_Lip}.
\end{proof}


\subsection{The modified a/c stress}
The function $\psiac(\mF;\bullet)$ obtained in Lemma \ref{th:acstress}
provides the divergence-free corrector for $\Sac - \Sa$ for
homogeneous deformations. We now construct the corrector for nonlinear
deformations: First, for each $f \in \Fsi$, $f = T_+ \cap T_-$, we set
\begin{displaymath}
 \mF_f(y) := \smfrac12 \b( \Dc{T_+} y +
 \Dc{T_-} y \b).
\end{displaymath}
We can now define the corrector function for $y \in \Ys_0$ as
\begin{equation}
  \label{eq:2d:defn_modwh}
  \psiy(y; \bullet) := \sum_{f \in \Fsi} \psiac\b( \mF_f(y); m_f \b) \zeta_f
  + \sum_{f \in \Fsc} \psic\b(y; m_f\b) \zeta_f,
\end{equation}
and the corresponding modified stress function
\begin{equation}
  \label{eq:2d:defn_Sacm}
  \Sacm(y; T) := \Sac(y; T) - \D \psiy(y; T) \mJ, 
  \qquad \text{for } T \in \T.
\end{equation}
We show in Remark \ref{rem:correctors_necessary}, that $\psiy$ is
non-trivial, that is, there exists no choice of parameters for which
$\psiy = 0$, even under purely homogeneous deformations.

The properties of the modified stress function $\Sacm$ are summarized
in the following lemma.

\begin{lemma}
  \label{th:acstress}
  Let $\Sacm$ be defined by \eqref{eq:2d:defn_Sacm}, and $y \in
  \Ys_0$; then the following identities hold:
  \begin{align}
    \label{eq:sacm_representation}
    \< \del\Eac(y), u \> =~& \sum_{T \in \T} |T|
    \Sacm(y; T) : \Dc{T} u \qquad \forall u \in \Us_0; \\
    \label{eq:Sacm_eq_sa}
    \Sacm(y;T) =~& \Sa(y;T) \qquad \forall T \in \Ta; \\
    \label{eq:Sacm_eq_Sc2}
    \Sacm(y;T) =~& \Sc^2(y;T) \qquad \forall T \in \Tc; \quad \text{and} \\
    \label{eq:Sacm_hom}
    \Sacm(\yF; \bullet) =~& \Sa(\yF; \bullet) \qquad \forall \mF \in
    \R^{2 \times 2}.
  \end{align}
  Moreover, there exists a constant $\hat{L}_\ac$, which depends only on
  $\bar{C}$, such that
  \begin{equation}
    \label{eq:Sacm_lip}
    \b| \Sacm(y; T) - \Sacm(\yF; T) \b| \leq \hat{L}_\ac M_2 \|D^2 y
    \|_{\ell^\infty(T \cap \L)}
    \qquad \forall T \in \Ti, \quad \mF = \Dc{T}y.
  \end{equation}
\end{lemma}


\begin{proof}
  Identity \eqref{eq:sacm_representation} follows from
  \eqref{eq:Sac_represenation} and the fact that $\Sacm - \Sac$ is
  divergence-free.

  Identity \eqref{eq:Sacm_eq_sa} follows from \eqref{eq:Sac_eq_sa} and
  the fact that $\psiy(y; m_f) = 0$ for all $f \in \Fsa$, which
  implies that $\Sacm(y; T) = \Sac(y; T) = \Sa(y; T)$ for all $T \in
  \Ta$. Similarly, \eqref{eq:Sacm_eq_Sc2} follows from~\eqref{eq:Sac_eq_Sc3}, and the fact that $\psiy = \psic$ in all
  elements $T \in \Tc$.

  Fix $\mF \in \R^{2 \times 2}$.  To prove \eqref{eq:Sacm_hom} we
  first note that, since $\psic(\yF; \bullet) = 0$, we have
  $\psiy(\yF; \bullet) = \psiac(\mF; \bullet)$. Using
  \eqref{eq:defn_psi_1}, we obtain
  \begin{displaymath}
    \Sacm(\yF; \bullet) = \Sac(\mF; \bullet) - \D\psiac(\mF; \bullet)
    \mJ = \Sa(\yF; \bullet).
  \end{displaymath}

  We are only left to prove the Lipschitz property
  \eqref{eq:Sacm_lip}. With $\mF := \Dc{T} y$, we have
  \begin{equation}
    \label{eq:Sacm_lip:1}
    \b| \Sacm(y; T) - \Sacm(\yF; T) \b| \leq 
    \b| \Sac(y; T) - \Sac(\yF; T) \b| + \b|\D\psiy(y; T) - \D \psiy(\yF; T)\b|.
  \end{equation}
  From its definition \eqref{eq:defn_Sac}, and the fact that second
  partial derivatives of $V$ are globally bounded, it is clear that
  $\Sac$ satisfies a Lipschitz property of the form
  \begin{equation}
    \label{eq:Sacm_lip:2}
    \b| \Sac(y; T) - \Sac(\yF; T) \b| \leq L_1 M_2  \| D^2 y
    \|_{\ell^\infty(T \cap \L)}
  \end{equation}
  where $L_1$ depends only on $\bar{C}$; see also \cite[Lemma
  19]{Ortner:2011:patch} for a similar result.  (If the reconstruction
  parameters satisfy the one-sidedness condition
  \eqref{eq:simple_Cxji}, as well as the patch test conditions
  \eqref{eq:energy_cons}, \eqref{eq:force_cons}, one may show that
  $L_1=3\bar{C}/\Omega_0$.)

  To bound the second term on the right-hand side in
  \eqref{eq:Sacm_lip:1} we invoke the inverse inequality
  \begin{displaymath}
    \b|\D\psiy(y; T) - \D \psiy(\yF; T)\b| \leq \frac{2}{\Omega_0}
    \sum_{\substack{f \in \Fs \\f \subset T}} \b| \psiy(y; m_f) -
    \psiy(\yF; m_f) \b|,
  \end{displaymath}
  where we used the fact that $|\D\zeta_f| = 2/\Omega_0$ for all $f
  \in \Fs$. If $f \in \Fsa$, then $\psiy(\bullet; m_f) = 0$. If $f \in
  \Fsc$, then $\psiy(\bullet; m_f) = \psic(\bullet; m_f)$ and hence,
  using \eqref{eq:psic_lipprop},
  \begin{displaymath}
     \b| \psiy(y; m_f) -
    \psiy(\yF; m_f) \b| = \b| \psic(y; m_f) \b| \leq \smfrac16 M_2 \|
    D^2 y \|_{\ell^\infty(T \cap \L)}.
  \end{displaymath}
  If $f \in \Fsi$, then $\psiy(y; m_f) = \psiac(\mF_f; m_f)$ and
  $\psiy(\yF; m_f) = \psiac(\mF; m_f)$. We can therefore employ
  \eqref{eq:psi_Lip} to estimate
  \begin{align*}
    \b| \psiy(y; m_f) -
    \psiy(\yF; m_f) \b| =~& \b| \psiac(\mF_f; m_f) -
    \psiac(\mF; m_f) \b| \\
    \leq~& L_\ac M_2 \b| \mF_f - \mF \b|
    \leq \smfrac{L_\ac M_2}{2 \Omega_0} \| D^2 y \|_{\ell^1(T \cap \L)}.
  \end{align*}
  The last inequality can be verified through straightforward
  geometric arguments. Without explicit constants its validity is
  obvious.

  Combining the two foregoing estimates, we obtain
  \begin{equation}
    \label{eq:Sacm_lip:3}
    \b|\D\psiy(y; T) - \D \psiy(\yF; T)\b| \leq   c \max(L_\ac, 1) M_2 \| D^2 y \|_{\ell^\infty(\L \cap T)}.
  \end{equation} 
  Combining \eqref{eq:Sacm_lip:1}, \eqref{eq:Sacm_lip:2} and
  \eqref{eq:Sacm_lip:3}, yields \eqref{eq:Sacm_lip}.
\end{proof}

\subsection{Proof of Theorem  \ref{th:2d:stress_err_Is}}
\label{sec:main_proof}
With the preparations of the foregoing sections it is now easy to
complete the proof of the main consistency result, Theorem
\ref{th:2d:stress_err_Is}. Again, we drop the dependence on $y$
whenever possible. We begin by splitting the consistency error into a
continuum contribution and an interface contribution,
\begin{align*}
  \< \del \Eac - \del \Ea, u \> 
  =~& \sum_{T\in\T} |T|\b[\Sacm(T)-\Sa(T)\b]:\Dc{T} u\\
  =~& \sum_{T\in\Tc} |T|\b[\Sacm(T)-\Sa(T)\b]:\Dc{T} u + \sum_{T\in\Ti}
  |T|\b[\Sacm(T)-\Sa(T)\b]:\Dc{T} u \\
  =:~& {\rm E}_\Cs + {\rm E}_\Is,
\end{align*}
and estimate ${\rm E}_\Cs$ and ${\rm E}_\Is$ separately. Note also
that we used \eqref{eq:Sacm_eq_sa} to drop the sum over elements in
the atomistic region.

Using the fact that $\Sacm = \Sc^2$ in $\Tc$, \eqref{eq:Sacm_eq_Sc2},
and the stress estimate \eqref{eq:stress_cb_o2}, we obtain
\begin{align}
  \notag
  {\rm E}_\Cs \leq~& \sum_{T \in \Tc} |T| \b| \Sc^2(T) - \Sa(T)
  \b| \, |\Dc{T} u| \\
  \notag
  \leq~& c \B( \sum_{x \in \Cs} \b[M_2 |D^3 y(x)| + 
  M_3 |D^2 y(x)|^2 \b]^p \B)^{1/p} \B( \sum_{T \in \Tc} |T| |\Dc{T} u|^{p'}\B)^{1/p'}
  \\
  \label{eq:est_ECs}
  \leq~& c \B( M_2 \|D^3 y\|_{\ell^p(\Cs)} 
  + M_3 \|D^2 y \|_{\ell^{2p}(\Cs)}^{2} \B)\,
  \B( \sum_{T \in \T} |T| |\Dc{T} u|^{p'}\B)^{1/p'}.
\end{align}

To estimate ${\rm E}_\Is$, we employ the Lipschitz property
\eqref{eq:Sacm_lip} for $\Sacm$ and the fact that $\Sac = \Sa$ under
homogeneous deformations (see \eqref{eq:Sacm_hom}). Using
\eqref{eq:Lip_partialV} it is also straightforward to prove
\begin{equation}
  \label{eq:lip_Sa}
  \b| \Sa(y; T) - \Sa(\yF; T) \b| \leq \smfrac{1}{\Omega_0} M_2 \| D^2
  y \|_{\ell^\infty(T \cap \L)} \qquad \forall T \in \T,  \quad \mF =
  \Dc{T} y.
\end{equation}
Using \eqref{eq:Sacm_hom}, \eqref{eq:Sacm_lip}
and \eqref{eq:lip_Sa},we obtain, for any $T \in \Ti$,
\begin{align*}
  \b| \Sacm(y; T) - \Sa(y; T) \b| \leq~& \b| \Sacm(y; T) - \Sacm(\yF; T)
  \b| + \b| \Sa(y; T) - \Sa(\yF; T) \b| \\
  \leq~& \hat{L}_\ac M_2 \| D^2 y \|_{\ell^\infty(T \cap \L)} + \smfrac{1}{\Omega_0} M_2
  \| D^2 y \|_{\ell^\infty(T \cap \L)},
\end{align*}
and summing over $T \in \Ti$ yields
\begin{align}
  \notag
  {\rm E}_\Is \leq~& \sum_{T \in \Ti} |T| (\hat{L}_\ac + \smfrac{1}{\Omega_0}) M_2 \|D^2
  y \|_{\ell^\infty(T \cap \L)} \, |\Dc{T} u| \\
  \label{eq:est_EIs}
  \leq~& c C_\Is M_2 \|D^2y\|_{\ell^p(\Isext)}
  \B( \sum_{T \in \T} |T| \,|\Dc{T} u |^{p'} \B)^{1/p'},
\end{align}
where $C_\Is$ depends only on $\hat{L}_\ac$, which depends only on
$\bar{C}$.

Combining \eqref{eq:est_EIs} and \eqref{eq:est_ECs} we finally obtain
the desired consistency error estimate \eqref{eq:main_cons_est}. This
concludes the proof of Theorem \ref{th:2d:stress_err_Is}.


\section{Appendix: Proof of Lemma \ref{th:ac_stress} (ii)}
In this appendix, we provide the remaining details for the proof of
\mbox{Lemma \ref{th:ac_stress} {\it (ii)}}. Throughout
this proof we fix a homogeneous deformation $\yF$, $\mF \in \R^{2
  \times 2}$, and drop the argument $y = \yF$ whenever possible. For
example, we will write $\Sac(T) = \Sac(\yF; T)$.

We begin by computing an expression for $\Sac - \Sa$ in terms of the
parameters $C_{x,j}$. Equation \eqref{eq:flat:facprf:5}, in the proof
of Lemma \ref{th:fac_v1}, can be rewritten in the form
\begin{align*}
  \Vac_{x,\mF,j} := \pp_j \Vac_x(\mF {\bf a})
  =(1-C_{x,j-1})V_{\mF,j-1}+C_{x,j}V_{\mF,j}+(1-C_{x,j+1})V_{\mF,j+1}.
\end{align*}
Recalling also \eqref{eq:defn_Sa} and using $a_j = a_{j-1} +
a_{j+1}$, we obtain
\begin{align}
  \notag
  \Sac(T)-~&\Sa(T) = \sum_{j=1}^6 \b[ \Vac_{x_{T,j},\mF,j} -
  V_{\mF, j} \b] \otimes a_j  \\
  \notag
  =~&
  \frac{1}{\Omega_0}\sum_{j=1}^6\b[(1-C_{x_{T,j},j-1})V_{\mF,j-1}+(C_{x_{T,j},j}-1)V_{\mF,j}+(1-C_{x_{T,j},j+1})V_{\mF,j+1}\b]\otimes
  a_j\\
  \notag
  =~&
  \frac{1}{\Omega_0}\sum_{j=1}^6V_{\mF,j}\otimes\b[(1-C_{x_{T,j-1},j})a_{j-1}+(C_{x_{T,j},j}-1)a_j+(1-C_{x_{T,j+1},j})a_{j+1}\b]\\
  \label{eq:Sac-Sa:hom}
  =~&
  \frac{1}{\Omega_0}\sum_{j=1}^6V_{\mF,j}\otimes\b[-C_{x_{T,j-1},j}a_{j-1}+C_{x_{T,j},j}a_j-C_{x_{T,j+1},j}a_{j+1}\b].
\end{align}
The explicit evaluation of \eqref{eq:Sac-Sa:hom} for interface
elements is carried out separately for flat interfaces and interfaces
with corners.

For triangles not intersecting the interface,
$\Sac(y_\mF;T)-\Sa(y_\mF;T)=0$, hence we need to compute the stress
errors only for interface elements.

\subsection{Flat interface}
\label{sec:flat:stress}
%
%

Consider the flat interface configuration in Figure
\ref{fig:interfacetriangle}. According to
\eqref{eq:interface_C_econs1} and \eqref{eq:interface_C_econs2} the
free parameters are $c_j := C_{x,j}$ (for $x \in \Is$ and $j \in
\{2,3,5,6\}$), and $d_i, i \in \Z$, where $d_1 = C_{x_{T_2,1}, 1} =
C_{x_{T_2,4},4}$, and so forth. We calculate the a/c stress for the
elements $T_1, T_2$, and collect the results in Table
\ref{tab:coefficients}.

\begin{figure}
  \begin{center}
    \includegraphics[height= 0.25\textwidth]{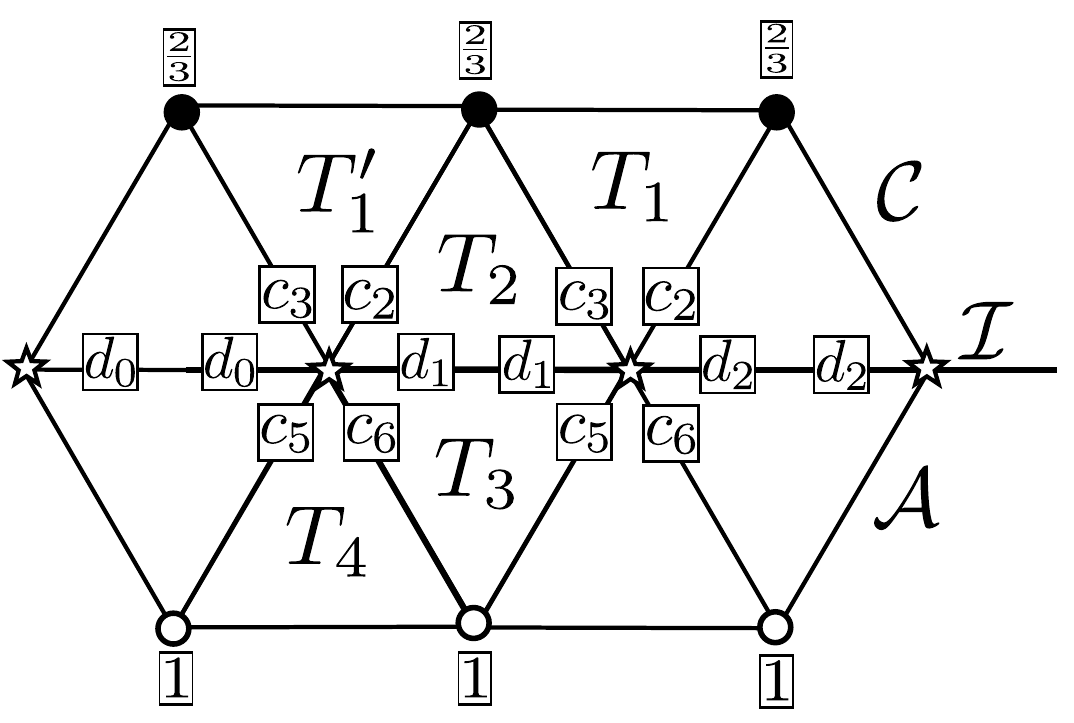}
   \end{center}
   \caption{Visualisation of the flat interface analysis in
     \S\ref{sec:flat:stress}.}
\label{fig:interfacetriangle}
\end{figure}

\begin{table}
  \caption{Table of coefficients of $V_{\mF, j}$ in
    \eqref{eq:Sac-Sa:hom}, in interfacial triangles, on flat
    interfaces.} \footnotesize
\begin{tabular}{|c|c|c|c|c|rl|}
\hline
&$j$&$C_{x_{T,j-1},j}$&$C_{x_{T,j},j}$&$C_{x_{T,j+1},j}$&
\multicolumn{2}{|c|}{$-C_{x_{T,j-1},j}a_{j-1}
    +C_{x_{T,j},j}a_j-C_{x_{T,j},j+1}a_{j+1}$} \\
\hline
\multirow{6}{*}{$T_1$}
& 1 & $\frac{2}{3}$ & $\frac{2}{3}$ & $d_2$         &
$-\frac{2}{3}a_6+\frac{2}{3}a_1-d_2a_2=$&\!\!\!\!$(\frac{2}{3}-d_2)a_2$
\\
& 2 & $\frac{2}{3}$ & $c_2$         & $c_2$         &
$-\frac{2}{3}a_1+c_2a_2-c_2a_3=$&\!\!\!\!$(-\frac{2}{3}+c_2)(a_2-a_3)$ \\
& 3 & $c_3$         & $c_3$         & $\frac{2}{3}$ &
$-c_3a_2+c_3a_3-\frac{2}{3}a_4=$&\!\!\!\!$(\frac{2}{3}-c_3)(a_2-a_3)$  \\
& 4 & $d_1$         & $\frac{2}{3}$ & $\frac{2}{3}$ &
$-d_1a_3+\frac{2}{3}a_4-\frac{2}{3}a_5=$&\!\!\!\!$(\frac{2}{3}-d_1)a_3$\\
& 5 & $\frac{2}{3}$ & $\frac{2}{3}$ & $\frac{2}{3}$ &
&\!\!\!\! 0                                                          \\
& 6 & $\frac{2}{3}$ & $\frac{2}{3}$ & $\frac{2}{3}$ &
&\!\!\!\! 0                                                          \\
\hline
\multirow{6}{*}{$T_2$}
& 1 & $\frac{2}{3}$ & $d_1$         & $d_1$         &
$-\frac{2}{3}a_6+d_1a_1-d_1a_2=$&\!\!\!\!$(\frac{2}{3}-d_1)a_3$        \\
& 2 & $c_2$         & $c_2$         & $c_2$         &
&\!\!\!\! 0                                                          \\
& 3 & $c_3$         & $c_3$         & $c_3$         &
&\!\!\!\! 0                                                          \\
& 4 & $d_1$         & $d_1$         & $\frac{2}{3}$ &
$-d_1a_3+d_1a_4-\frac{2}{3}a_5=$&\!\!\!\!$(\frac{2}{3}-d_1)a_2$        \\
& 5 & $c_5$         & $\frac{2}{3}$ & $\frac{2}{3}$ &
$-c_5a_4+\frac{2}{3}a_5-\frac{2}{3}a_6=$&\!\!\!\!$(c_5-\frac{2}{3})a_1$\\
& 6 & $\frac{2}{3}$ & $\frac{2}{3}$ & $c_6$         &
$-\frac{2}{3}a_5+\frac{2}{3}a_6-c_6a_1=$&\!\!\!\!$(\frac{2}{3}-c_6)a_1$ \\
\hline
\end{tabular}
\label{tab:coefficients}
\end{table}

From Table \ref{tab:coefficients} we can read off the stress
differences $\Sac-\Sa$ in the elements $T_1, T_2$:
\begin{align*}
  \Sac(T_1)-\Sa(T_1) =  \b\{(\smfrac{2}{3}-d_2)V_{\mF,1}
  +(c_2-\smfrac{2}{3})V_{\mF,2}+(\smfrac{2}{3}-c_3)V_{\mF,3}\b\}&
  \otimes \smfrac{a_2}{\Omega_0} \\
   + \b\{(d_1-\smfrac{2}{3})V_{\mF,1}+(\smfrac{2}{3}-c_2)V_{\mF,2}
  +(c_3-\smfrac{2}{3})V_{\mF,3}\b\}&
  \otimes \smfrac{a_3}{\Omega_0},
\end{align*}
and
\begin{align*}
  \Sac(T_2)-\Sa(T_2)=\b\{(d_1-\smfrac{2}{3})V_{\mF,1}
  +(\smfrac{2}{3}-c_5)V_{\mF,2}+(c_6-\smfrac{2}{3})V_{\mF,3}\b\}
  & \otimes \smfrac{a_1}{\Omega_0} \\
  =
  \b\{(d_1-\smfrac{2}{3})V_{\mF,1}+(\smfrac{2}{3}-c_2)V_{\mF,2}
  +(c_3-\smfrac{2}{3})V_{\mF,3}\b\} & \otimes
  \smfrac{a_2}{\Omega_0} \\
  -
  \b\{(d_1-\smfrac{2}{3})V_{\mF,1}+(\smfrac{2}{3}-c_2)V_{\mF,2}
  +(c_3-\smfrac{2}{3})V_{\mF,3}\b\} & \otimes \smfrac{a_3}{\Omega_0} \\
  + \b\{(c_2-c_5)V_{\mF,2}+(c_6-c_3)V_{\mF,3}\b\} &
  \otimes \smfrac{a_1}{\Omega_0}.
\end{align*}
Note that we have provided two alternative representations of
$\Sac(T_2)-\Sa(T_2)$, since the first representation is in general
insufficient to construct the corrector.

Since the atomistic region is a mirror image of the continuum region
with respect to the interface, we can obtain stress function
$\Sac(y_\mF;\cdot)$ for $T_3$ and $T_4$ from symmetry considerations:
\begin{align*}
  \Sac(T_4)-\Sa(T_4) =
  \b\{(1-d_0)V_{\mF,1}+(c_5-1)V_{\mF,2}+(1-c_6)V_{\mF,3}\b\} 
  & \otimes \smfrac{a_2}{\Omega_0} \\  
   + \b\{(d_1-1)V_{\mF,1}+(1-c_5)V_{\mF,2}+(c_6-1)V_{\mF,3}\b\}
  & \otimes \smfrac{a_3}{\Omega_0},
\end{align*}
and
\begin{align*}
  \Sac(T_3)-\Sa(T_3)=\b\{(d_1-1)V_{\mF,1}+(1-c_2)V_{\mF,2}+(c_3-1)V_{\mF,3}\b\}
  & \otimes \smfrac{a_1}{\Omega_0}\\
  =\b\{(d_1-1)V_{\mF,1}+(1-c_5)V_{\mF,2}+(c_6-1)V_{\mF,3}\b\}
  & \otimes \smfrac{a_2}{\Omega_0} \\
  -\b\{(d_1-1)V_{\mF,1}+(1-c_5)V_{\mF,2}+(c_6-1)V_{\mF,3}\b\}
  & \otimes \smfrac{a_3}{\Omega_0} \\
  -\b\{(c_2-c_5)V_{\mF,2}+(c_6-c_3)V_{\mF,3}\b\} & \otimes \smfrac{a_1}{\Omega_0}.
\end{align*} 

From the proof Lemma \ref{th:psic} recall that $\D\zeta_f(T) \mJ = -
\smfrac{2}{\Omega_0} a_j$ if $f$ is an edge of $T$ and $a_j$ the
counter-clockwise direction of the edge (relative to $T$). We can
therefore choose $\psiac$ explicitly, for example, for $f=T_1\cap
T_2$:
\begin{equation}
  \label{eq:psiac_1}
  \psiac(\mF;m_f) := \smfrac12 \b\{(d_1-\smfrac{2}{3}) V_{\mF,1} +
  (\smfrac{2}{3}-c_2) V_{\mF,2} + (c_3-\smfrac{2}{3}) V_{\mF,3} \b\}.
\end{equation}
For the remaining edges, similar choices can be made, the crucial
observation being that the terms in neighbouring elements associated
with an edge cancel each other out.

We observe, moreover, that for the triangles $T_1$ and $T_4$, the
$a_1$ components of the stresses vanish, which means that
$\psiac(\mF;m_f)=0$ for all $f\in\Fsa\cup\Fsc$. This proves
\eqref{eq:defn_psi_2} in the flat interface case.

It remains to prove the Lipschitz bound \eqref{eq:psi_Lip}. From
\eqref{eq:psiac_1} (and the corresponding formulas for the remaining
edges), it is straightforward to show that $\psiac$ is Lipschitz
continous for any fixed set of parameters with a Lipschitz constant of
the form $L M_2$, where $L$ can be bounded in terms of $\bar{C}$. This
concludes the proof of Lemma \ref{th:ac_stress} (ii) in the flat
interface case.


\begin{remark}[Correctors are neccessary]
  \label{rem:correctors_necessary}
  From the calculation in this section, it is clear that one cannot
  choose parameters such that $\Sac(\yF;T)=\Sa(\yF;T)$ for all $T \in
  \T$ and for all potentials $V \in \Vs$. For example, if $\Sac(\yF;
  T_2) = \Sa(\yF; T_2)$ for all $V$, then $d_1 = 2/3$, whereas if
  $\Sac(\yF; T_3) = \Sa(\yF; T_3)$, then $d_1 = 1$. This demonstrates
  that the divergence-free corrector fields are in fact necessary, and
  that it is impossible in our current framework to construct an a/c
  method where $\Sac(\yF; T) = \Sa(\yF; T)$ holds for all $T \in \T, \mF
  \in \R^{2\times 2}$, and $V \in \Vs$.
\end{remark}

\subsection{General interface}
\label{sec:corner:stress}
We now turn to the proof of \eqref{eq:defn_psi_1}--~\eqref{eq:psi_Lip}
for interface configurations with corners. Consider the corner
configuration displayed in Figure \ref{fig:cornertriangle}, which is
concave from the point of view of the atomistic region. The
reconstruction coefficients found in
Proposition~\ref{th:params:corner} are displayed in the figure as
well. Recall that the reconstructions of bonds into the atomistic or
continuum regions are now uniquely determined, while the bonds lying
at the interfaces (parameters $a$ and $b$) are still free.

\begin{figure}
  \begin{center}
    \includegraphics[width=6cm]
    {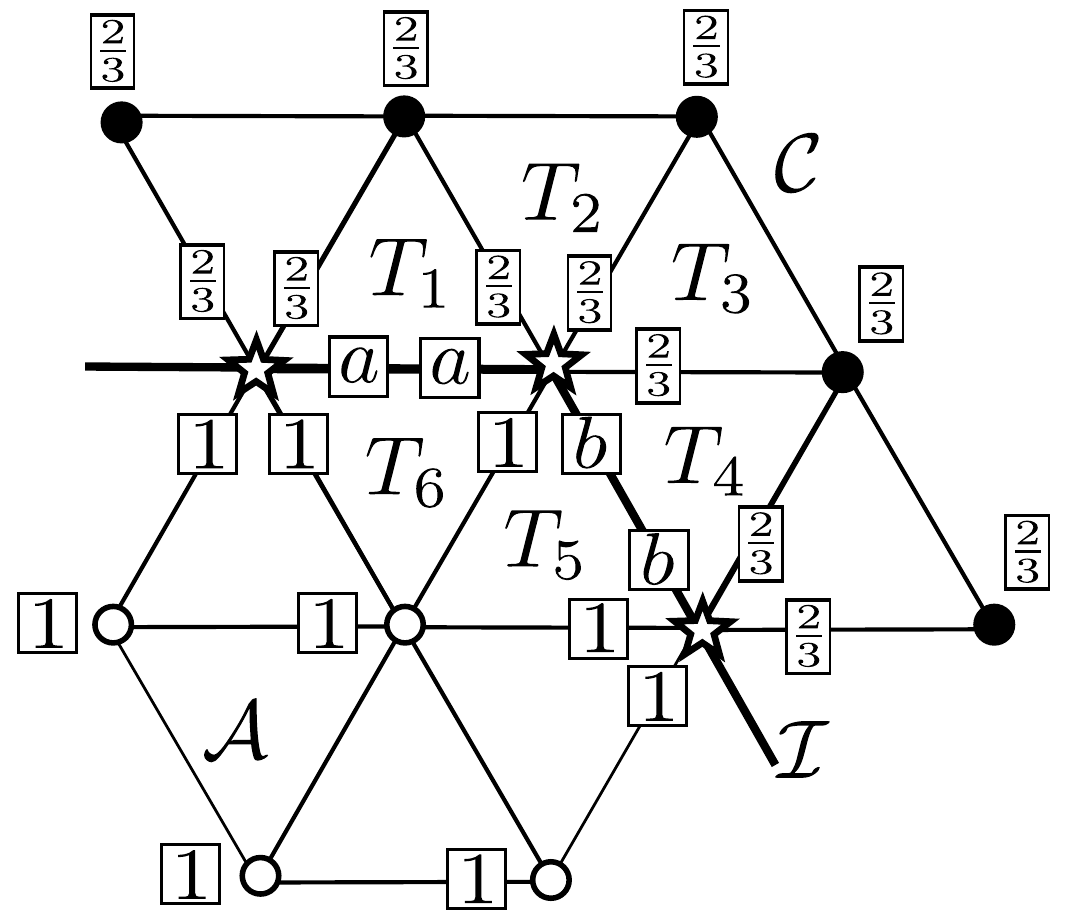}
   \end{center}
   \caption{Interface configuration with corner.}
\label{fig:cornertriangle}
\end{figure}

Using \eqref{eq:Sac-Sa:hom}, and defining $a_j' := a_j / \Omega_0$,
the stress errors $\Sac - \Sa$ in the elements $T_1, \dots, T_6$ can
again be computed explicitly:
\begin{align*}
  \Sac(T_1)-\Sa(T_1)=~& (\smfrac{1}{3}V_{\mF,3}
  -\smfrac{1}{3}V_{\mF,2}) \otimes a_1' + (a-\smfrac{2}{3})V_{\mF,1}\otimes a_2'-(a-\smfrac{2}{3})V_{\mF,1}\otimes a_3',\\
  \Sac(T_2)-\Sa(T_2)=~&(a-\smfrac{2}{3})V_{\mF,1}\otimes a_3',\\
  \Sac(T_3)-\Sa(T_3)=~&(b-\smfrac{2}{3})V_{\mF,3}\otimes a_1',\\
  \Sac(T_4)-\Sa(T_4)=~&-(b-\smfrac{2}{3})V_{\mF,3}\otimes a_1'+(b-\smfrac{2}{3})V_{\mF,3}\otimes a_2'+(\smfrac{1}{3}V_{\mF,1}-\smfrac{1}{3}V_{\mF,2})\otimes a_3',\\
  \Sac(T_5)-\Sa(T_5)=~&(\smfrac{1}{3}V_{\mF,2}-\smfrac{1}{3}V_{\mF,1})\otimes
  a_3' +[(1-a)V_{\mF,1}+(b-1)V_{\mF,3}]\otimes a_2' \\
  & - (b-1)V_{\mF,3}\otimes a_1', \quad \text{and} \\
  \Sac(T_6)-\Sa(T_6)=~&(\smfrac{1}{3}V_{\mF,2}-\smfrac{1}{3}V_{\mF,3})\otimes
  a_1'+\b[(a-1)V_{\mF,1}+(1-b)V_{\mF,3}\b]\otimes a_2'\\
  & -(a-1)V_{\mF,1}\otimes a_3'.
\end{align*}
Following the argument in \S\ref{sec:flat:stress}, we can check again
that the associated edge contributions from neighbouring elements
cancel, and hence we can explicitly construct the corrector function
$\psiac$. Note that $\Sac(T_2)$ has no $a_1'$ component and
$\Sac(T_5)$ has no $a_3'$ component, which implies
\eqref{eq:defn_psi_2}.

For a corner that is convex from the point of view of the atomistic
region, the result follows by symmetry (interchanging the coefficients
$1$ and $\smfrac23$). The Lipschitz bound \eqref{eq:psi_Lip} can be
obtained from the above formulas, under the assumption that the
reconstruction coefficients $a, b$ are bounded above by $\bar{C}$.


Finally, we have to convince ourselves that our above argument applies
to all possible interface geometries. In Figure \ref{fig:corners} we
present an exhaustive list, up to translations, rotations and
reflections, of local interface geometries. (Recall our geometric
requirements formulated in Assumption \ref{th:geninterface}.) By
inspecting the calculation of the stress differences $\Sac - \Sa$ for
the case presented in Figure \ref{fig:cornertriangle}, one observes
that the formulas are local, and do not depend on the extended
geometry of the interface. We note, however, that this only holds due
to the separation Assumption \ref{th:geninterface}. The subsequent
construction of the corrector now follow of course verbatim.

\begin{figure}
  \begin{center}
    \includegraphics[width=0.6\textwidth]{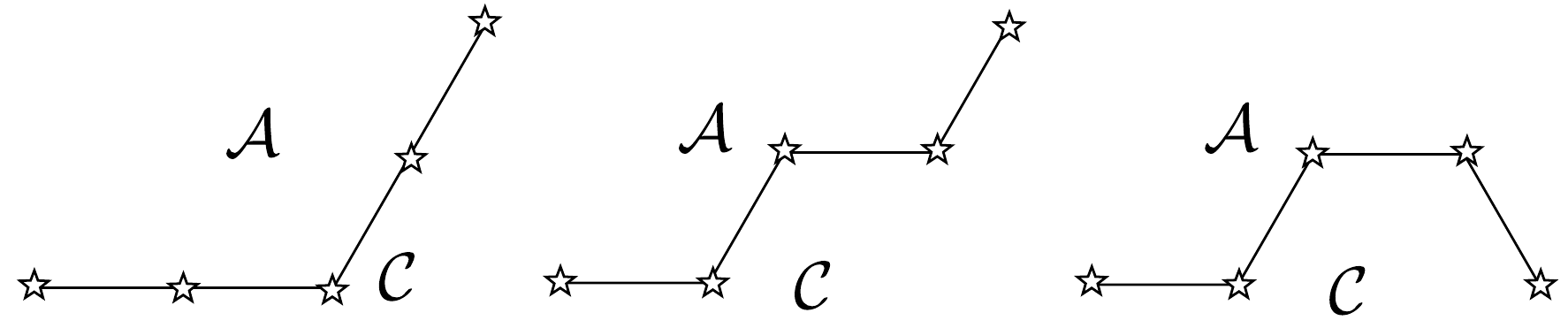}
   \end{center}
   \caption{All possible corner configurations (up to translation,
     rotation and reflection).}
\label{fig:corners}
\end{figure}

This concludes the proof of Lemma \ref{th:ac_stress} {\it (ii)} in the
general interface case.

\section{Conclusion}
We have shown for a 2D model problem that it is possible to construct
patch test consistent a/c coupling method for multi-body potentials,
in interface geometries with corners, using a new variant of the
geometry reconstruction technique introduced in \cite{E:2006,
  Shimokawa:2004}, which we labelled the GR-AC method. Moreover, we
have proven a quasi-optimal consistency error estimate for the GR-AC
method(s) we constructed.

We see this work as a first step towards a general theory of GR-AC
method(s). Our goal is to show eventually that the free parameters in the
method can {\it always} (that is, in any dimension, for any interface
geometry) be determined so as to satisfy the energy and force
consistency conditions, and that the resulting GR-AC method(s) will
have the same consistency properties that we establish in the present
case.

An important issue that we have left entirely open in the present work
is the stability of the GR-AC method: Under which conditions on the
reconstruction parameters does the GR-AC method have sharp stability
properties as discussed in \cite{Dobson:qce.stab}?  This issue is the
topic of ongoing research.

\bibliographystyle{plain}
\bibliography{qc}

\end{document}